\documentclass[11pt]{article}
\usepackage{smile}

\usepackage{fullpage}
\usepackage{framed}
\usepackage{mdframed}
\usepackage{enumerate}
\usepackage[inline]{enumitem}
\usepackage[T1]{fontenc}
\usepackage{moresize}
\usepackage{bm}
\usepackage{dsfont}

\usepackage{amsmath}
\usepackage{amssymb}
\usepackage{amsthm}
\usepackage{amsfonts}
\usepackage{stmaryrd}
\usepackage{array}
\usepackage{mathrsfs}
\usepackage{mathtools} 
\usepackage{extarrows}
\usepackage{stackrel}
\usepackage{physics}

\usepackage[nodisplayskipstretch]{setspace}
\usepackage{color}
\usepackage[usenames,dvipsnames]{xcolor}
\usepackage{cancel}
\usepackage{soul}
\usepackage{xfrac}
\usepackage{siunitx}
\usepackage{booktabs}
\usepackage{multirow}
\usepackage{siunitx}
\usepackage{graphicx}
\graphicspath{ {./images/} }
\usepackage{float}
\usepackage{rotating}
\usepackage{subcaption}
\usepackage{overpic}
\usepackage[all]{xy}
\usepackage{tikz}
\usetikzlibrary{arrows,matrix,positioning,calc,automata,patterns}
\usepackage{booktabs}
\usepackage{dcolumn}
\usepackage{multirow}
\usepackage{diagbox}
\usepackage{lscape}
\usepackage{rotating}
\usepackage{adjustbox}

\usepackage{verbatim}
\usepackage{listings}
\usepackage{fancyvrb}
\usepackage{hyperref}
\usepackage[round]{natbib}
\usepackage{sectsty}

\usepackage[linesnumbered,ruled]{algorithm2e}
\usepackage{algpseudocode}
\SetKwInput{KwInput}{Input}             
\SetKwInput{KwInitialization}{Initialization}      
\SetKwInput{KWReturn}{Return}

\hypersetup{
    bookmarks=true,         
    unicode=false,          
    pdftoolbar=true,        
    pdfmenubar=true,        
    pdffitwindow=false,     
    pdfstartview={FitH},    
    pdftitle={My title},    
    pdfauthor={Author},     
    pdfsubject={Subject},   
    pdfcreator={Creator},   
    pdfproducer={Producer}, 
    pdfkeywords={key1, key2}, 
    pdfnewwindow=true,      
    colorlinks=true,        
    linkcolor=blue,         
    citecolor=blue,         
    filecolor=blue,         
    urlcolor=cyan           
}

\def\tr{\mathop{\text{tr}}\kern.2ex}

\def\ind{{\mathds 1}}

\newcolumntype{L}[1]{>{\raggedright\let\newline\\\arraybackslash\hspace{0pt}}m{#1}}
\newcolumntype{C}[1]{>{  \centering\let\newline\\\arraybackslash\hspace{0pt}}m{#1}}
\newcolumntype{R}[1]{>{ \raggedleft\let\newline\\\arraybackslash\hspace{0pt}}m{#1}}
\newcolumntype{d}[1]{D{.}{.}{#1}}
\newcolumntype{H}{>{\setbox0=\hbox\bgroup}c<{\egroup}@{}}
\newcolumntype{Z}{>{\setbox0=\hbox\bgroup}c<{\egroup}@{\hspace*{-\tabcolsep}}}

\numberwithin{equation}{section}
\newtheorem{theorem}{Theorem}[section]
\newtheorem{lemma}{Lemma}[section]
\newtheorem{proposition}{Proposition}[section]
\newtheorem{assumption}{Assumption}[section]

\newtheorem{innercustomgeneric}{\customgenericname}

\providecommand{\customgenericname}{}
\newcommand{\newcustomtheorem}[2]{%
  \newenvironment{#1}[1]
  {%
   \renewcommand\customgenericname{#2}%
   \renewcommand\theinnercustomgeneric{##1}%
   \innercustomgeneric
  }
  {\endinnercustomgeneric}
}
\newcustomtheorem{customtheorem}{Theorem}
\newcustomtheorem{customassumption}{Assumption}
\newcustomtheorem{customlemma}{Lemma}
\newcustomtheorem{customexample}{Example}
\theoremstyle{definition}

\newtheorem{remark}{Remark}[section]

\DeclareFontFamily{U}{mathx}{\hyphenchar\font45}
\DeclareFontShape{U}{mathx}{m}{n}{<-> mathx10}{}
\DeclareSymbolFont{mathx}{U}{mathx}{m}{n}
\DeclareMathAccent{\widebar}{0}{mathx}{"73}

\begin{document}

\setlength{\abovedisplayskip}{5pt}
\setlength{\belowdisplayskip}{5pt}
\setlength{\abovedisplayshortskip}{5pt}
\setlength{\belowdisplayshortskip}{5pt}
\hypersetup{colorlinks,breaklinks,urlcolor=blue,linkcolor=blue}

\title{\LARGE A sliced Wasserstein and diffusion approach to random coefficient models}

\author{
Keunwoo Lim\thanks{Department of Applied Mathematics, University of Washington, Seattle, WA 98195, USA. E-mail: \tt{kwlim@uw.edu}}, ~~Ting Ye\thanks{Department of Biostatistics, University of Washington, Seattle, WA 98195, USA. E-mail: \tt{tingye1@uw.edu}}, ~ and~ Fang Han\thanks{Department of Statistics, University of Washington, Seattle, WA 98195, USA; e-mail: {\tt fanghan@uw.edu}}}

\date{\today}

\maketitle


\begin{abstract} 
We propose a new minimum-distance estimator for linear random coefficient models. This estimator integrates the recently advanced sliced Wasserstein distance with the nearest neighbor methods, both of which enhance computational efficiency. We demonstrate that the proposed method is consistent in approximating the true distribution. Moreover, our formulation naturally leads to a diffusion process-based algorithm and is closely connected to treatment effect distribution estimation—both of which are of independent interest and hold promise for broader applications.
\end{abstract}


\section{Introduction} \label{section: introduction}

Consider a linear random coefficient model (RCM) involving a set of random variables $(X, Y, \beta)$ with the underlying structure  
\begin{align}\label{eq:linear-RCM}
    Y = \langle \beta, X \rangle := \text{the inner product of }X \text{ and }\beta.
\end{align}  
Here the covariates $X$ and the random coefficients $\beta$ are both supported on $\mathbb{R}^d$. Throughout this paper, we assume that $X$ and $\beta$ are independent. This paper focuses on estimating the distribution of $\beta$ based solely on the observations $\{(X_i, Y_i); i \in [n]\}$, where $[n] := \{1, 2, \ldots, n\}$ and $\{(X_i, Y_i, \beta_i); i \in [n]\}$ are $n$ independent realizations of $(X, Y, \beta)$.  

Linear RCMs of the form \eqref{eq:linear-RCM} play a fundamental role in many quantitative studies; see, for instance, \cite{lewbel2017unobserved} and \cite{Bonhomme2024} for some recent reviews on the applications of RCMs in economics. Traditionally, research on RCMs has primarily focused on estimating the mean and variance of $\beta$ \citep{rubin1950note, hildreth1968some, swamy1970efficient}. While these two moments summarize average marginal effects and variances, they often fail to capture all information of interest to researchers.  

To address this limitation, Beran and Hall \citep{beran92} pioneered the study of nonparametrically identifying and estimating the distribution of $\beta$. Since then, various approaches have been proposed, including minimum-distance estimation \citep{Beran76,beran1993semiparametric}, series expansion methods for (conditional) density functions \citep{beran1996nonparametric, Hohmann2016, Dunker19, Dunker25, gaillac2021nonparametric, Gaillac22}, kernel methods \citep{Hoderlein2010, Holzmann20}, nearest neighbor (NN) methods \citep{holzmann2024multivariate}, and discrete-grid-based constrained least squares methods \citep{fox2011, fox2016simple, heiss2022nonparametric}.  

In this paper, we revisit the minimum-distance estimators proposed by Beran and Millar \citep{Beran76} and enhance their framework by incorporating nearest neighbor (NN) methods \citep{lin2023estimation,holzmann2024multivariate} and the sliced Wasserstein (SW) metric, the latter of which has recently gained prominence in generative artificial intelligence. Compared to generative modeling approaches based on the original Wasserstein metric \citep{arjovsky2017wasserstein, tolstikhin2017wasserstein, bousquet2017optimal}, the SW metric has been observed to provide similar empirical performance while being computationally more attractive and less sensitive to the dimension of the data \citep{kolouri2018sliced, liutkus2019sliced,fan2024minimum}. Computationally, optimization algorithms towards minimizing the SW metric have also been introduced and validated in \cite{bonneel2015sliced} and \cite{tanguy2024properties}, among many others. 

Building on these insights, this paper proposes a novel SW-based minimum-distance estimator for the distribution of $\beta$, where weights are adaptively chosen over an NN graph to reduce bias. To the best of our knowledge, this formulation is new. Computationally, following the framework of \cite{tanguy2024properties}, we develop block gradient descent algorithms that achieve efficient optimization with a time complexity that scales polynomially in $d$, whereas existing methods typically require an exponentially large number of grid points in $d$. We further establish theoretical guarantees for the proposed estimator and derive its rates of convergence. Lastly, drawing inspiration from parallel research on SW flows \citep{bonnotte2013unidimensional, liutkus2019sliced} and causal random coefficient models \citep{heckman1997making}, we bridge our formulation with diffusion processes and causal inference. For the former, we introduce diffusion process-based methods to approximate the distribution of $\beta$. For the latter, we incorporate  entropic regularization into the model of \cite{heckman1997making}, enhancing the identification and estimation of the treatment effect distribution.

\vspace{0.2cm}
\noindent {\bf Paper organization.} Section \ref{section: minimum-distance estimator}  presents the proposed minimum-distance estimator and Section \ref{section: algorithm} illustrates the (approximate) block coordinate decent algorithms for implementation. Section \ref{section: theory} presents the main theory concerning estimation and computation. Sections \ref{section: diffusion} and \ref{sec: causal} connect the studied problem to diffusion processes and casual inference, respectively. The empirical performance of the algorithms is reported in Section \ref{section: simulations}. Proofs of all theoretical results are relegated to the Appendix.

\vspace{0.2cm}
\noindent{\bf Notation.} Let  $(\mathbb{R}^{d}, \|\cdot\|_2)$ represent the $d$-dimensional real space coupled with the Euclidean metric. Write $\mathbb{S}^{d-1} = \{V \in \mathbb{R}^{d}\,\vert\, \Vert V \Vert_{2}=1\}$, $B_{r}(a) = \{ U\in \mathbb{R}^{d} \,\vert \,\Vert U - a\Vert_{2}< r \}$, and  $\widebar{B}_{r}(a)$ to represent the sphere, the radius-$r$ open ball, and the corresponding closed ball in $(\mathbb{R},\|\cdot\|_2)$, respectively. Let $\{e_{1},\dots, e_{d}\}$ be the standard bases in $\mathbb{R}^{d}$ with regard to the Euclidean inner product $\langle\cdot,\cdot\rangle$, and let $\sigma$ be the Haar measure in $\mathbb{S}^{d-1}$. For any vectors $v_1,\ldots,v_m\in\mathbb{R}^d$, write $(v_1,\ldots,v_m)=(v_1^\top,\ldots,v_m^\top)^\top$. For vectors $U, \tilde{U}\in \mathbb{R}^{d}$, we also use $P^{U}(\tilde{U}) = \langle U, \tilde{U}\rangle$ to represent their inner product. Let $\mathcal{P}(K)$ denote the set of all probability measures on $K$ and $\delta_{x}\in \mathcal{P}(K)$ be the Dirac measure on $x \in K$. Lastly, for arbitrary measure $\mu \in \mathcal{P}(\mathbb{R}^{d})$, we introduce its projected measure with respect to the direction $V\in\mathbb{S}^{d-1}$ as $ \mu^{V} = \mu \circ (P^{V})^{-1}$.

\section{Method}\label{section: minimum-distance estimator}

Consider $\{(X_i, Y_i, \beta_i); i \in [n]\}$ to be $n$ realizations of \eqref{eq:linear-RCM}. The objective of interest is to estimate the distribution of $\beta$, denoted as $\mu_\beta$, based solely on the observations $(X_i, Y_i)$'s.

This paper is concerned with minimum-distance estimation \citep{Wolfowitz53, Wolfowitz57}, which involves an objective functional and a domain of minimization. In our formulation, the objective functional is based on the Wasserstein $W_2$ distance, defined as  
\begin{align*}
    W_2(\mu, \tilde{\mu}) = \bigg( \inf_{\pi \in \Pi(\mu, \tilde{\mu})} \mathbb{E}_{(x, y) \sim \pi} \Vert x - y \Vert_2^2 \bigg)^{\frac{1}{2}},
\end{align*}
where $\mu$ and $\tilde{\mu}$ are two probability measures on $\mathbb{R}^d$, and $\Pi(\mu, \tilde{\mu})$ denotes the set of all couplings of $\mu$ and $\tilde{\mu}$, i.e., probability measures on $\mathbb{R}^d \times \mathbb{R}^d$ with marginal distributions $\mu$ and $\tilde{\mu}$, respectively.

For each direction $V \in \mathbb{S}^{d-1}$, we introduce a $k$-NN empirical measure $\widebar{\mu}^V_k$ for the objective functional. Specifically, for each $i \in [n]$, first let  
\[
    \tilde{X}_{i} = \frac{X_{i}}{\Vert X_{i} \Vert_{2}}, \quad 
    \tilde{Y}_{i} = \frac{Y_{i}}{\Vert X_{i} \Vert_{2}}
\]
be the {\it normalized} counterparts of $X_i$ and $Y_i$, respectively. Let $\ind(\cdot)$ represent the indicator function and 
\[
    \overline{S}_X(V, k) := \Big\{ j \in [n]: \sum_{i=1}^n \ind\big(\| \tilde{X}_i - V \|_2 < \| \tilde{X}_j - V \|_2 \big) < k \Big\}  
\]
denote the set of indices corresponding to the $k$-NNs of $V$ in $\{\tilde{X}_i; i \in [n]\}$. Since $X_i$ may take discrete values, ties can occur, causing the cardinality of $\overline{S}_X(V, k)$ to exceed $k$. In such cases, consistent with \cite{lin2024failure, lin2024consistency}, we select an {\it arbitrary} subset 
\[
    S_X(V, k) \subset \overline{S}_X(V, k) \quad \text{such that its cardinality}~ \big| S_X(V, k) \big| = k.  
\]
When $\tilde{X}_i$ is continuous, as discussed in Section~\ref{subsection: estimation}, we have $S_X(V,k) = \overline{S}_X(V,k)$.

Using this subset and noting that, for any direction $V \in \mathbb{S}^{d-1}$ and $i \in [n]$,
\begin{align}\label{eq:insight}
    \beta_i^\top V \approx \tilde{Y}_i \tilde{X}_i^\top V \quad \text{as } \tilde{X}_i \approx V,
\end{align}
we define the NN-induced random measure as  
\[
    \widebar{\mu}^V_k = \frac{1}{k} \sum_{i \in S_X(V, k)} \delta_{P^V(\tilde{Y}_i \tilde{X}_i)} 
    = \frac{1}{k} \sum_{i \in S_X(V, k)} \delta_{\tilde{Y}_i \tilde{X}_i^\top V}.
\]
In light of \eqref{eq:insight}, for arbitrary $V$ and all sufficiently small $k$, the laws of $\widebar{\mu}^V_k$ and $\beta^\top V$ should be close to each other.

Encouraged by the above insight, we introduce the following objective functional
\begin{align}\label{eq:fk}
    \mathcal{F}_k(\mu) = \int_{\mathbb{S}^{d-1}} W_2^2(\widebar{\mu}^V_k, \mu^V) \, \mathrm{d} \sigma(V),
\end{align}
which integrates the squared $W_2$ distance between the NN random measure $\widebar{\mu}^V_k$ and the projected measure $\mu^V$ over all directions $V \in \mathbb{S}^{d-1}$. 

Lastly, we propose the following {\it constrained} $k$-NN-based minimum-distance estimator
\begin{align}\label{eq:mde}
    \hat{\mu}_\beta \in \argmin_{\mu \in \mathcal{P}_N(\widebar{B}_R(0))} \mathcal{F}_k(\mu).
\end{align}
Here $R$ is a pre-specified radius to enforce a compact space, and the domain $\mathcal{P}_N(\widebar{B}_R(0))$ consists of all size-$N$ discrete measures in $\mathcal{P}(\widebar{B}_R(0))$, defined as  
\begin{align*}
    \mathcal{P}_N(\widebar{B}_R(0)) = \Big\{ \eta(w) \in \mathcal{P}(\widebar{B}_R(0)) : \eta(w) = \frac{1}{N} \sum_{i=1}^N \delta_{w_i} \text{ with } w_1, \ldots, w_N \in \widebar{B}_R(0)\\
    \text{and } w= (w_1, \ldots, w_N)\in \widebar{B}_R(0)^N \Big\}.
\end{align*}
The output is then a discrete approximation to any $\mu_\beta\in \mathcal{P}(\widebar{B}_R(0))$.

\begin{remark}
There are connections between \eqref{eq:fk} and the sliced Wasserstein $W_2$ distance. Notice that the sliced $W_2$ distance between any two probability measures, $\mu, \tilde{\mu} \in \mathcal{P}(\mathbb{R}^d)$, is defined as  
\begin{align*}
    SW_2(\mu, \tilde{\mu}) = \bigg( \int_{\mathbb{S}^{d-1}} W_2^2(\tilde\mu^V, \mu^V) \, \mathrm{d} \sigma(V) \bigg)^{\frac{1}{2}}.
\end{align*}
Thus, $\mathcal{F}_k(\cdot)$ can be interpreted as a $k$-NN-based revision of the original $SW_2$ distance, encouraged by the approximate equality in \eqref{eq:insight}. In particular, when $k = n$, $\mathcal{F}_k(\cdot)$ reduces to the $SW_2$ distance between the empirical measure and the target.  
\end{remark}

\section{Algorithm}\label{section: algorithm}

The computation of $\hat{\mu}_{\beta}$ in \eqref{eq:mde} reduces to finding a size-$N$ set of elements in $\widebar{B}_R(0)$:
\begin{align}\label{eq:points}
\hat{w} \in \eta^{-1}(\hat{\mu}_{\beta}).
\end{align}
Optimizing \eqref{eq:mde}  can thus also be interpreted as minimizing $\mathcal{F}_{k} \circ \eta$ over $\widebar{B}_{R}(0)^{N}$, which is identifiable up to a permutation. However, solving this minimization problem requires continuous integration over the Haar measure, which is computationally infeasible.

To address this issue, we adopt a common trick and propose a discrete Monte Carlo approximation to the original $\mathcal{F}_k$:
\begin{align}\label{eq:approximate-mde}
    \widebar{\mathcal{F}}_{k}(\mu) = \frac{1}{m} \sum_{i = 1}^{m} W_{2}^{2}(\widebar{\mu}_{k}^{V_{i}}, \mu^{V_{i}}),
\end{align}
where $V_{1}, \dots, V_{m}$ are sampled independently from the Haar measure $\sigma$, independent of the system. 

We accordingly introduce two algorithms for optimizing \eqref{eq:mde} using the Monte Carlo approximation \eqref{eq:approximate-mde}, namely, Algorithms \ref{algorithm: bcd} and \ref{algorithm: abcd}. Without loss of generality, in these two algorithms we set $N = k$ since otherwise, we could choose the larger of the two. In addition, in these algorithms, we use functions for {\it sorting, argsorting}, and {\it projecting} and define them as follows. For a real sequence $a = (a_{q})_{q\in[k]} = (a_{1}, \dots, a_{k})$, we define
\[
\text{sort}(a) = (a_{\nu(q)})_{q\in[k]} = (a_{\nu(1)}, \dots, a_{\nu(k)})
\]
and
\[
\text{argsort}(a) = (\nu(q))_{q\in[k]} = (\nu(1), \dots, \nu(k)),
\]
where $(\nu(1), \dots, \nu(k))$ is an arbitrary permutation of $[k]$ such that $a_{\nu(1)} \leq \cdots \leq a_{\nu(k)}$. Additionally, the projection of a vector $\psi \in \mathbb{R}^{kd}$ onto the set $\widebar{B}_{R}(0)^{k}$ in Euclidean space is defined as
\[
\text{Proj}_{\widebar{B}_{R}(0)^{k}}(\psi) = \argmin_{w \in \widebar{B}_{R}(0)^{k}} \Vert \psi - w \Vert_{2}.
\]

Algorithm \ref{algorithm: bcd} is a block coordinate descent method for computing $\hat{w}$. This algorithm is a revised version of \citet[Algorithm 1]{tanguy2024properties}, incorporating the $k$-NN graph into the computation. In each iteration, the local optimal transport law is updated using Lemma \ref{lemma: Bobkov} in the Appendix, leveraging the vector $\psi_{\ell}$ and the $k$-NN matrix $D$. The vector $\psi_{\ell+1}$ is obtained as the minimizer of the local optimal transport law within the domain $\widebar{B}_{R}(0)^{k}$. 

The computation of $\psi_{\ell+1}$, however, involves convex optimization over a constrained domain, which can be computationally expensive. Algorithm \ref{algorithm: abcd} addresses this computational challenge by introducing a projected gradient descent variant. In this algorithm, the optimal transport law from Algorithm \ref{algorithm: bcd} is approximated using an $L^{2}$-distance formulation. This enables the minimization problem to be solved using linear algebra, significantly reducing computational complexity. In addition, it is possible to check that the computational complexity is $O(mdn\log n+tmk \log k)$; see Appendix Section \ref{sec:cc-calculation} ahead for details. This is substantially faster than the existing grid-based algorithms that has computational complexity exponentially growing with the dimension.

\begin{algorithm}[!ht]
\caption{Block Coordinate Descent}\label{algorithm: bcd}
\LinesNotNumbered
  \KwInput{$(Y_{i}, \beta_{i}, X_{i})$, $V_{j}$, $R$, $i\in[n]$, $q\in[m]$}
  \KwInitialization{$\psi_{1} \sim \text{Uniform}\,(\widebar{B}_{R}(0)^{k})$, $\psi_{0}$, $\ell \leftarrow 1$, $D\in \mathbb{R}^{k \times m}$}
  \For{$j = 1$ to $m$}
  {Choose the $k$-NNs of $V_{j}$ as $\{\tilde{X}_{j(q)}; q\in[k]\}$ with ties broken arbitrarily\\
   $D_{\cdot,\, j} \leftarrow \text{sort}\,(\, \tilde{Y}_{j(q)}\,\langle\, V_{j},\,\tilde{X}_{j(q)}\,\rangle\,)$ for each $q\in[k]$
  }
  \While{$\Vert \psi_{\ell} - \psi_{\ell-1}\Vert_{2} > 0$}
  {
  \For{$j = 1$ to $m$}
  {$\nu_{j}\leftarrow \text{argsort}\,(\langle V_{j}, \psi_{\ell, 1}\rangle,\ldots,\langle V_{j}, \psi_{\ell, k}\rangle)$
  }
  $\psi_{\ell+1}\leftarrow\argmin_{w\in \widebar{B}_{R}(0)^{k}}\sum_{j = 1}^{m}\sum_{q = 1}^{k}(\langle V_{j}, w_{q}\rangle - D_{\nu_{j}^{-1}(q),\, j})^{2}$\\
  $\ell \leftarrow \ell+1$} 
 \KWReturn{$\psi_{\ell}$}
\end{algorithm}

\begin{algorithm}[!ht]
\caption{Approximate Block Coordinate Descent}
\label{algorithm: abcd}
\LinesNotNumbered
  \KwInput{$(Y_{i}, \beta_{i}, X_{i})$, $V_{j}$, $R$, $t$, $i\in[n]$, $q\in[m]$}
  \KwInitialization{$\psi_{1} \sim \text{Uniform}\,(\widebar{B}_{R}(0)^{k})$, $D\in \mathbb{R}^{k \times m}$}
  \For{$j = 1$ to $m$}
  {Choose the $k$-NNs of $V_{j}$ as $\{\tilde{X}_{j(q)}; q\in[k]\}$ with ties broken arbitrarily\\
   $D_{\cdot,\, j} \leftarrow \text{sort}\,(\, \tilde{Y}_{j(q)}\,\langle\, V_{j},\,\tilde{X}_{j(q)}\,\rangle\,)$ for each $q\in[k]$
  }
  \For{$\ell = 1$ to $t$}
  {\For{$j = 1$ to $m$}
  {$\nu_{j}\leftarrow \text{argsort}\,(\langle V_{j}, \psi_{\ell, 1}\rangle,\ldots,\langle V_{j}, \psi_{\ell, k}\rangle)$
  }
  \For{$q = 1$ to $k$}
  {$\psi_{\ell+1, q}\leftarrow m^{-1}d\sum_{j=1}^{m}D_{\nu_{j}^{-1}(q), \,j}V_{j}$
  }
$\psi_{\ell+1}\leftarrow \text{Proj}_{\widebar{B}_{R}(0)^{k}}(\psi_{\ell+1})$\\
  }
 \KWReturn{$\psi_{t+1}$}
\end{algorithm}

\section{Theory}\label{section: theory}

This section is divided into three parts. The first part discusses the assumptions, while the other two address the rates of convergence of the estimator and the convergence of the proposed algorithms.

\subsection{Assumptions}

We focus on the case $d \geq 2$, as the case $d=1$ is trivial. Define the normalized versions of $X$ and $Y$ as $\tilde{X} = X / \Vert X \Vert_{2}$ and $\tilde{Y} = Y / \Vert X \Vert_{2}$, respectively, such that $\tilde{X} \in \mathbb{S}^{d-1}$. This normalization preserves the inner product structure, as $\tilde{Y} = \langle \beta, \tilde{X} \rangle$.

The first assumption specifies the data-generating structure, which  implies that $\{(\tilde{X}_i, \tilde{Y}_i, \beta_i); i \in [n]\}$ are independent and identically distributed as $(\tilde{X}, \tilde{Y}, \beta)$. 

\begin{assumption}\label{assumption:1}
Assume $\{(X_i, Y_i, \beta_i); i \in [n]\}$ are $n$ independent realizations of $(X, Y, \beta)$, which satisfies the linear structure \eqref{eq:linear-RCM} and that $X$ is independent of $\beta$.
\end{assumption}

The next assumption imposes constraints on the distribution of $\beta$, ensuring that it has a compact support. This setting is similar to those made in \cite{Beran76}, \cite{Hoderlein2010}, and \cite{Holzmann20}, and also common in other nonparametric mixture model analyses \citep{Miao24,Han23,lim2024smoothed}.

\begin{assumption}\label{assumption: coefficient}
There exists a finite constant $R < \infty$ such that  $\mu_{\beta}$ is supported on $\widebar{B}_{R}(0)$.
\end{assumption}

The following two assumptions provide restrictions on the distribution of $\tilde{X}$.

\begin{assumption}\label{assumption: bounded covariates}
The law of $\tilde{X}$ is absolutely continuous with respect to the Haar measure $\sigma$ on $\mathbb{S}^{d-1}$. Furthermore, the corresponding Haar density, denoted by $f_{\tilde{X}}$, is lower bounded by some positive constant $\tau_{0} > 0$.
\end{assumption}

Assumption \ref{assumption: bounded covariates} is equivalent to Assumption 4 in \cite{Hoderlein2010}, which naturally holds when $X$ has a full-dimensional support. In particular, this assumption is valid when \eqref{eq:linear-RCM} does not include an intercept term, i.e., $X_{i,1} = 1$ for all $i \in [n]$. Such scenarios are plausible in many cases, including those discussed in \citet[Section 4.3]{Hoderlein2010} and \cite{lewbel2017unobserved}. However, it no longer holds if the model includes an intercept term. To address such cases, we introduce the following general assumption.

\begin{assumption}\label{assumption: unbounded covariates}
The law of $\tilde{X}$ is absolutely continuous with respect to the Haar measure $\sigma$ on $\mathbb{S}^{d-1}$. Additionally, there exist constants $\tau_{0}$, $\rho_{0}$, $C_{\tilde{X}}$, and $\alpha > 0$ such that for all $0 < \tau \leq \tau_{0}$ and $0 < \rho \leq \rho_{0}$, there exist closed sets $D(\tau)$, $L(\rho, \tau) \subset \mathbb{S}^{d-1}$ satisfying:
\begin{align*}
    \sigma\Big(\Big\{\, V \in \mathbb{S}^{d-1} \,\Big|\, f_{\tilde{X}}(V) \leq \tau \,\Big\} \setminus D(\tau) \Big) = 0, \quad
    L(\rho, \tau) = \Big\{\, V \in \mathbb{S}^{d-1} \,\Big|\, B_{\rho}(V) \cap D(\tau) = \emptyset \,\Big\},
\end{align*}
and 
\[
\sup_{\tilde{V} \in \mathbb{S}^{d-1} \setminus L(\rho, \tau)} \min_{V \in L(\rho, \tau)} \Vert V - \tilde{V} \Vert_{2} \leq C_{\tilde{X}} \tau^{\alpha} + \rho.
\]    
\end{assumption}

Assumption~\ref{assumption: unbounded covariates} is designed to be weak, although at the expense of some technical complexity. To aid understanding, we present the following result, which simplifies Assumption~\ref{assumption: unbounded covariates} to a tail condition on $X$.

\begin{proposition}\label{prop:heavy-tail}
Suppose there exist constants $C_{f} > 0$ and $\kappa > 1$ such that $X = (1, X_{2:d})$ and the density $f_{X_{2:d}}$ of $X_{2:d}$ satisfies:
\[
f_{X_{2:d}}(T) \geq C_{f}(1 + \Vert T \Vert_{2})^{-\kappa}, 
\]
for all $T \in \mathbb{R}^{d-1}$. Assumption \ref{assumption: unbounded covariates} then holds with $\alpha = 1/(\kappa + 1)$, $\tau_{0} = \mathcal{S}(\mathbb{S}^{d-1})C_{f}/2^{3\kappa + 3}$, and $\rho_{0} = 1/4$, where $\mathcal{S}(\mathbb{S}^{d-1})$ is a surface area of $\mathbb{S}^{d-1}$.
\end{proposition}

\begin{remark}
Proposition \ref{prop:heavy-tail} refines Assumption 1 in \cite{Holzmann20} and similar assumptions in \cite{Dunker19} and \cite{Dunker25} (see, e.g., the discussions in \citet[Section 4.3]{Dunker25}) by accommodating a broader range of tail conditions. In the literature, such tail conditions are often introduced for technical convenience, as we also do here. Notably, in an interesting recent paper, \cite{Gaillac22} examined some cases where $X$ is even allowed to have a compact support, albeit at the cost of imposing certain supersmoothness conditions on the density of $\beta$.
\end{remark}



\subsection{Estimation}\label{subsection: estimation}

Our first main theoretical result addresses the landscape of the objective function and the existence of minimizers to \eqref{eq:mde}.

\begin{theorem}\label{proposition: convex}
Under Assumption \ref{assumption: coefficient}, the random function $\mathcal{F}_{k} \circ \eta$ over the domain $\widebar{B}_{R}(0)^{N}$ is $4RN^{-1/2}$-Lipschitz continuous. Additionally, with probability one, there exist global minimizer(s) for \eqref{eq:mde}, which is unique (up to a permutation) if $N = k = n$.
\end{theorem}  

The next theorem establishes the rates of convergence of $\hat\mu_\beta$ to $\mu_\beta$ under the $SW_2$ distance. The first part pertains to the case under Assumption \ref{assumption: bounded covariates}, and the second part addresses the case under Assumption \ref{assumption: unbounded covariates}.

\begin{theorem}\label{theorem: bounded estimator}
Assume Assumptions \ref{assumption:1} and \ref{assumption: coefficient}. Then, \textit{any} minimum-distance estimator $\hat{\mu}_{\beta}$ of \eqref{eq:mde} satisfies the following.
\begin{itemize}
\item[(i)] Under Assumption \ref{assumption: bounded covariates}, there exists a positive constant $C = C(R, d, \tau_{0})$ such that
\begin{align*}
\begin{cases}
    \mathbb{E}[SW_{2}(\mu_{\beta}, \hat{\mu}_{\beta})] \leq C n^{-1/(d+5)}, &\text{for } 2 \leq d \leq 5\text{ with } N = k = n^{\frac{6}{d+5}}, \\
    \mathbb{E}[SW_{2}(\mu_{\beta}, \hat{\mu}_{\beta})] \leq C n^{-1/(2d-1)}, &\text{for } d \geq 6 \text{ with } N = k = n^{\frac{d}{2d-1}}.
\end{cases}
\end{align*}
\item[(ii)] Under Assumption \ref{assumption: unbounded covariates}, there exists a positive constant $C = C(R, d, \tau_{0}, \rho_{0}, C_{\tilde{X}})$ such that
\begin{align*}
\begin{cases}
    \mathbb{E}[SW_{2}(\mu_{\beta}, \hat{\mu}_{\beta})] \leq C n^{-\frac{\alpha}{2d\alpha+10\alpha+4}}\log n, &\text{for } 2 \leq d \leq 5 \text{ with } N = k = n^{\frac{6\alpha}{d\alpha+5\alpha+2}}, \\
    \mathbb{E}[SW_{2}(\mu_{\beta}, \hat{\mu}_{\beta})] \leq C n^{-\frac{\alpha}{4d\alpha-2\alpha+4}}\log n, &\text{for } d \geq 6 \text{ with } N = k = n^{\frac{d\alpha}{2d\alpha-\alpha+2}}.
\end{cases}
\end{align*}
\end{itemize}
\end{theorem}

It is worth noting that Theorem \ref{proposition: convex} guarantees the existence of global minimizer(s) for \eqref{eq:mde}, but it does not provide a general guarantee of uniqueness unless $N = k = n$. 
Theorem \ref{theorem: bounded estimator} further demonstrates that \textit{any} global minimizer of \eqref{eq:mde} achieves a polynomial rate of convergence to the true parameter $\mu_\beta$. For related discussions, see \citet[Remark 3.1]{Miao24}, which observe similar phenomena in nonparametric Poisson mixtures where multiple global minimizers (in that case, nonparametric maximum likelihood estimators) may exist but all converge to the estimand. 

At this stage, it remains unclear whether the derived rate is minimax optimal or not. Moreover, as different metrics are employed, the obtained results are not directly comparable to those in the literature. Nevertheless, it is worth highlighting that, unlike prior works including \cite{Hoderlein2010}, \cite{Holzmann20}, and \cite{Gaillac22}, our analysis avoids imposing {\it any} smoothness conditions on the densities of $X$ or $\tilde{X}$, which underscores an advantage of the proposed distance-based approach.

\subsection{Computation}\label{subsection: computation}

Section \ref{subsection: estimation} discusses the statistical properties of the estimator $\hat\mu_\beta$ in \eqref{eq:mde}. However, from a practical standpoint, it is also crucial to understand the procedure from a computational perspective, which is the focus of this section. In this section, the algorithms of interest for approximating $\hat\mu_{\beta}$ are Algorithms \ref{algorithm: bcd} and \ref{algorithm: abcd} presented in Section \ref{section: algorithm}.

Our first theoretical result in this section examines the landscape of the approximate objective function \eqref{eq:approximate-mde}, paralleling Theorem \ref{proposition: convex}.

\begin{proposition}\label{proposition: discrete continuity}
Under Assumption \ref{assumption: coefficient}, the random function $\widebar{\mathcal{F}}_{k}\circ \eta$ over the domain $\widebar{B}_{R}(0)^{N}$ is $4RN^{-1/2}$-Lipschitz continuous. Additionally, with probability one, there exist global minimizer(s) for \eqref{eq:approximate-mde}, which is unique (up to a permutation) if $N = k = n$.
\end{proposition}

Next, we aim to quantify the algorithmic convergence. To this end, we introduce the following proposition, which analyzes the approximation of $\widebar{\mathcal{F}}_{k}\circ\eta$ to $\mathcal{F}_{k}\circ\eta$.

\begin{proposition}\label{proposition: uniform convergence}
Under Assumption \ref{assumption: coefficient} and conditional on the data:
\begin{itemize}
\item[(i)] $\widebar{\mathcal{F}}_{k}\circ \eta$ converges to $\mathcal{F}_{k}\circ \eta$ uniformly over $\widebar{B}_{R}(0)^{N}$ in the sense that
\begin{align*}
    \mathbb{P}_{\sigma}\Big(\lim_{m\to\infty} \Big\Vert \widebar{\mathcal{F}}_{k}\circ\eta - \mathcal{F}_{k}\circ \eta \Big\Vert_{\infty} = 0\Big) = 1,
\end{align*}
where 
\[
\Big\Vert \widebar{\mathcal{F}}_{k}\circ\eta - \mathcal{F}_{k}\circ \eta \Big\Vert_{\infty} := \sup_{w\in \widebar{B}_{R}(0)^{N}}\Big\Vert \widebar{\mathcal{F}}_{k}\circ\eta(w) - \mathcal{F}_{k}\circ \eta(w) \Big\Vert;
\]

\item[(ii)] furthermore, 
\begin{align*}
    \sqrt{m} \Big\Vert \widebar{\mathcal{F}}_{k}\circ\eta - \mathcal{F}_{k}\circ\eta \Big\Vert_{\infty} ~~\text{converges in distribution to}~~ \Big\Vert \mathbb{G} \Big\Vert_{\infty},
\end{align*}
where $\mathbb{G}$ is a centered Gaussian process on $\widebar{B}_{R}(0)^{N}$ with covariance structure
\begin{align*}
    \textnormal{cov}\,\mathbb{G}(w, \tilde{w}) = \int_{V \in \mathbb{S}^{d-1}} W_{2}^{2}(\widebar{\mu}_{k}^{V}, \eta(w)^{V})W_{2}^{2}(\widebar{\mu}_{k}^{V}, \eta(\tilde{w})^{V}) \dd \sigma(V) - \mathcal{F}_{k}(\eta(w))\mathcal{F}_{k}(\eta(\tilde{w})). 
\end{align*}
\end{itemize}
\end{proposition}

The main theorem of this section establishes the validity of Algorithms \ref{algorithm: bcd} and \ref{algorithm: abcd}.

\begin{theorem}\label{theorem: computation}
Assuming $N = k$ and under Assumption \ref{assumption: coefficient}, Algorithm \ref{algorithm: bcd} outputs a local minimizer of $\widebar{\mathcal{F}}_{k}(\eta(w))$. Furthermore, in the case of $\psi_{\ell} = \tilde{\psi}_{\ell}$, where $\psi_{\ell}$ and $\tilde{\psi}_{\ell}$ denote the $\ell$-th iteration of Algorithms \ref{algorithm: bcd} and \ref{algorithm: abcd}, respectively, we have
\begin{align*}
    \lim_{m\to\infty} \mathbb{P}_{\sigma}\Big( W_{2}(\eta(\psi_{\ell+1}), \eta( \tilde{\psi}_{\ell+1})) > \epsilon \Big) = 0
\end{align*}
for all $\epsilon > 0$.
\end{theorem}



\section{A diffusion process approach}\label{section: diffusion}

In this section, building on a recent breakthrough connecting optimal transport-based generative modeling to diffusion processes \citep{liutkus2019sliced}, we propose a diffusion process-based approximation to the optimization problem \eqref{eq:mde}. While rigorous theoretical support for this approximation is currently unavailable, its empirical effectiveness has been studied and will be demonstrated in Section \ref{section: simulations}. 

Specifically, consider the following minimization problem involving a regularized functional $\tilde{\mathcal{F}}_{k}^{\lambda}(\mu)$ over the space of probability measures $(\mathcal{P}(\widebar{B}_{R}(0)), W_{2})$:
\begin{align*}
    \tilde{\mathcal{F}}_{k}^{\lambda}(\mu) := \frac{1}{2} \mathcal{F}_{k}(\mu) + \lambda \mathcal{H}(\mu),
\end{align*}
where $\mathcal{F}_{k}(\cdot)$ is introduced in \eqref{eq:fk}, $\lambda > 0$ is a tuning parameter, and $\mathcal{H}(\mu)$ is the regularization functional defined as:
\begin{align*}
    \mathcal{H}(\mu) = 
    \begin{cases}
        \int \varrho \log \varrho\,\dd x, & \text{if the density $\varrho$ of $\mu$ exists}, \\
        +\infty, & \text{otherwise}.
    \end{cases}
\end{align*}

The following proposition provides the theoretical background for formulating the associated continuity equation in this setting. For functions \(f\colon \mathbb{R}^{d} \to \mathbb{R}^{d}\) and \(\tilde{f}\colon \mathbb{R}^{d} \to \mathbb{R}\), define:
\[
\textnormal{div} f \coloneqq \sum_{i=1}^{d} \frac{\partial f_{i}}{\partial x_{i}}, \quad \Delta \tilde{f} \coloneqq \sum_{i=1}^{d} \frac{\partial^{2} \tilde{f}}{\partial x_{i}^{2}}.
\]

\begin{proposition}[Theorem 2, \cite{liutkus2019sliced}]
\label{proposition: flow}
Assume that \(\lambda > 0\), \(R > \sqrt{d}\), and let \(\mu_{0}\) be a probability measure in \(\mathcal{P}(\widebar{B}_{R}(0))\) with density \(\varrho_{0} \in L^{\infty}(\widebar{B}_{R}(0))\). Then, for the minimization problem:
\begin{align*}
    \tilde{\mathcal{F}}_{\zeta}^{\lambda}(\mu) = \frac{1}{2}SW_{2}^{2}(\mu, \zeta) + \lambda \mathcal{H}(\mu),
\end{align*}
where \(\zeta\) is a probability measure with positive smooth density, there exists a generalized minimizing movement scheme such that the density \((\varrho_{t})_{t}\) of \((\mu_{t})_{t}\) satisfies the continuity equation:
\begin{align}\label{eq: continuity}
    \frac{\partial \varrho_{t}}{\partial t} = -\textnormal{div}(v_{t} \varrho_{t}) + \lambda \Delta \varrho_{t},
\end{align}
in the weak sense \citep[Page 123]{santambrogio2015optimal}, where \(v_{t}\) is associated with the Kantorovich potential \(\Upsilon_{t}^{V}\) \citep[Page 13]{santambrogio2015optimal} between \(\mu_{t}^{V}\) and \(\zeta^{V}\), given by:
\begin{align*}
    v_{t}(x) = v(x, \mu_{t}) = -\int_{\mathbb{S}^{d-1}} (\Upsilon_{t}^{V})'\,(\langle x, V \rangle)\,V \dd \sigma(V).
\end{align*}
\end{proposition}

\begin{remark}
The regularization functional \(\mathcal{H}(\mu)\) plays a crucial role in generative modeling, as it is believed to mitigate overfitting. In contrast, when \(\lambda = 0\), the problem reduces to the original estimation task, with the corresponding continuity equation provided in, e.g., \citet[Theorem 5.6.1]{bonnotte2013unidimensional}. Furthermore, by defining \(\mathcal{H}(\mu)\), we restrict our focus to those probability measures \(\mu\) that are absolutely continuous with respect to the Lebesgue measure. While this excludes the discrete cases considered in Section~\ref{section: theory}, it aligns with the smoothness requirements of optimal transport maps \citep[Chapter 5.4.2]{bonnotte2013unidimensional}.  
\end{remark}

Equation \eqref{eq: continuity} resembles the Fokker-Planck equation \citep[Theorem 2.2]{pavliotis2014stochastic} for the stochastic differential equation
\begin{align}\label{eq: sde}
    \dd Q_{t} = v(Q_{t}, \mu_{t})\, \dd t + \sqrt{2\lambda}\,\dd W_{t},
\end{align}
where \((W_{t})_{t}\) is a standard Brownian motion. We then propose Algorithm \ref{algorithm: flow} to approximate \eqref{eq: sde} using particle systems \citep[Algorithm 1]{liutkus2019sliced} and Euler-Maruyama discretization \citep[Section 5.2]{pavliotis2014stochastic}. For \(L\) particles at step \(\ell\), the iteration is given by:
\begin{align}\label{eq: particle}
    Q_{(\ell+1)h}^{i} = Q_{\ell h}^{i} + v(Q_{\ell h}^{i}, \tilde{\mu}_{\ell h}^{L})\,h + \sqrt{2\lambda h}\,Z_{\ell h}^{i}, \quad 1 \leq i \leq L,
\end{align}
where \(Z_{\ell h}^{i}\) is standard Gaussian noise. The drift \(v(Q_{\ell h}^{i}, \tilde{\mu}_{\ell h}^{L})\) is computed as:
\begin{align}\label{eq: drift}
    v(Q_{\ell h}^{i}, \tilde{\mu}_{\ell h}^{L}) = -\frac{1}{m}\sum_{j=1}^{m} \left(\langle Q_{\ell h}^{i}, V_{j} \rangle - F_{\widebar{\mu}^{V}_{k}}^{-1} \circ F_{\tilde{\varrho}_{\ell h}^{V}}(\langle Q_{\ell h}^{i}, V_{j} \rangle)\right)V_{j},
\end{align}
where \(F_{\tilde{\varrho}_{\ell h}^{V}}\) is the empirical CDF of \((\tilde{\mu}_{\ell h}^{L})^{V}\), and \(F_{\widebar{\mu}^{V}_{k}}^{-1}\) is the empirical quantile function for \(\widebar{\mu}^{V}_{k}\). The algorithm outputs \(Q_{(t+1)h}^{i}\) as samples of the unobserved coefficient. If \(m\) scales polynomially with \(d\), the algorithm has polynomial complexity in \(d\). 

\begin{algorithm}[!ht]
\caption{Diffusion-based Generative Modeling}
\label{algorithm: flow}
\LinesNotNumbered
  \KwInput{$(Y_{i}, \beta_{i}, X_{i})$, $V_{j}$, $L$, $R$, $\lambda$, $h$, $t$, $i \in [n]$,  $ j \in [m]$}
  \KwInitialization{$\psi_{1} \sim \text{Uniform}\,(\widebar{B}_{R}(0)^{L})$, $D\in \mathbb{R}^{k \times m}$}
  \For{$j = 1$ to $m$}
  {Choose the $k$-NNs of $V_{j}$ as $\{\tilde{X}_{j(q)}; q\in[k]\}$ with ties broken arbitrarily\\
   $D_{\cdot,\, j} \leftarrow \text{sort}\,(\, \tilde{Y}_{j(q)}\,\langle\, V_{j},\,\tilde{X}_{j(q)}\,\rangle\,)$ for each $q\in[k]$
  }
  Construct $F_{\widebar{\mu}^{V}_{k}}^{-1}$ with $D$\\
  \For{$\ell = 1$ to $t$}
  {Construct $F_{\tilde{\varrho}_{\ell h}^{V}}$ with $\psi_{\ell}$\\
  \For{$i = 1$ to $L$}
  {$v(\psi_{\ell, i}, \tilde{\mu}_{\ell h}^{L}) \leftarrow -\frac{1}{m}\sum_{j = 1}^{m} (\, \langle \psi_{\ell, i}, V_{j}\rangle - F_{\widebar{\mu}^{V}_{k}}^{-1} \circ F_{\tilde{\varrho}_{\ell h}^{V}}(\langle \psi_{\ell, i}, V_{j} \rangle)\,)\,V_{j}$\\
  $\psi_{\ell+1, i} \leftarrow \psi_{\ell, i}+ v(\psi_{\ell, i}, \tilde{\mu}_{\ell h}^{L})\,h+ \sqrt{2\lambda h}\,Z_{\ell h}^{i}$
  }
  }
 \KWReturn{$\psi_{t+1}$}
\end{algorithm}

\section{Application to causal inference}\label{sec: causal}

In this section, we connect the model \eqref{eq:linear-RCM} to the causal random coefficient model introduced in \citet[Section 7]{heckman1997making}:
\begin{align}\label{eq:causal-RCM}
Y(w) = \langle Z, \beta_Z \rangle + w \cdot R + U, ~~~\text{where } (\beta_Z,R,U) \text{ is independent of } Z.
\end{align}
Here $Z \in \mathbb{R}^d$ represents pretreatment covariates, $w \in \mathcal{W} \subset \mathbb{R}$ is the (continuous or discrete) treatment, $U$ is an unspecified noise term, and $\{Y(w); w \in \mathcal{W} \}$ denotes the potential outcomes corresponding to different treatment levels. It follows that
\[
R = Y(w+1) - Y(w), \quad \text{for any } w \in \mathcal{W}.
\]
Thus, $R$ denotes the unit-level treatment effect. In the special case where $\mathcal{W} = \{0,1\}$, we obtain 
\[
R = Y(1) - Y(0),
\]
which represents the unit-level treatment effect of receiving the treatment compared to control. 

Notably speaking, we have made two key assumptions in Model \eqref{eq:causal-RCM}. First, we have assumed that the untreated potential outcome satisfies the random coefficient structure: $ Y(0) =  \langle Z, \beta_Z \rangle + U$.
Second, we have assumed that the unit-level treatment effect $R$ is independent of the observed covariates $Z$, though, it may be correlated with the unobserved noise $U$.


In practice, however, not all potential outcomes $\{Y(w); w \in \mathcal{W} \}$ are observable.	In the context of a {\it randomized controlled trial}, we introduce a random treatment variable $W \in \mathcal{W}$ and  only observe the outcome $Y : = Y(W)$, which corresponds to the potential outcome under the assigned treatment. The observed outcome follows the model:
\begin{align}\label{eq:causal-RCM-1}
Y(W) = \langle Z, \beta_Z \rangle + W \cdot R  + U, \quad \text{where } W \text{ is independent of } (Z,U,\beta_Z,R).
\end{align}
Clearly, Model \eqref{eq:causal-RCM-1} is a special case of \eqref{eq:linear-RCM}, with
\begin{align}\label{eq:causal-notation}
\beta = (\beta_Z, R, U), \quad X = (Z, W, 1).
\end{align}
This facilitates the estimation of $R$'s distribution in view of the results  established in previous sections. 
However, to make the arguments formal, we still need to verify Assumption \ref{assumption: unbounded covariates}, which is the focus of the following results.

\begin{assumption}\label{assumption: causal tail}
There exist constants $\kappa > 1$ and $C_{Z} > 0$ such that the Lebesgue density of $Z$, denoted by $f_Z$, satisfies
\begin{align*}
f_Z(T) \geq C_{Z}(1 + \Vert T \Vert_{2})^{-\kappa}, \quad \text{for all } T \in \mathbb{R}^d.
\end{align*}
\end{assumption}

The following proposition addresses the case of a continuous treatment $W$.

\begin{proposition}\label{prop: causal_heavy_tail} 
Assume Assumption \ref{assumption: causal tail}, and further assume that the Lebesgue density of $W$, denoted by $f_W$, satisfies
\[
f_W(t) \geq C_{W}(1 + |t|)^{-\kappa}, \quad \text{for any } t \in \mathbb{R}.
\]
Then, the random vector $X = (Z, W, 1)$ defined in \eqref{eq:causal-notation} satisfies
\[
f_{(Z, W)}(T, t)\geq C_{Z}C_{W} (1+\Vert(T, t)\Vert_{2})^{-2\kappa}.
\]
Thus, Assumption \ref{assumption: unbounded covariates} holds with $\alpha = 1/(2\kappa + 1)$, $\tau_{0} = \mathcal{S}(\mathbb{S}^{d-1})C_{Z}C_{W}/2^{6\kappa + 3}$, and $\rho_{0} = 1/4$, as established in Proposition \ref{prop:heavy-tail}.
\end{proposition}
This indicates that the Assumption \ref{assumption: unbounded covariates} holds for the case of continuous and unbounded treatment, which aligns with the support assumption regarding the identification in \citet[Theorem 1]{heckman1997making}. 

However, assuming a continuous and unbounded treatment may be restrictive in practice. To address this, we introduce a regularized \textit{working model}, analogous to ridge regression in linear models and entropic regularization in optimal transport:
\begin{align}\label{eq:working-model}
\text{(Regularized Working Model)} \quad Y_{\epsilon}(W) = \langle Z, \beta_Z \rangle + R \cdot W_{\epsilon} + U,
\end{align}
where
\[
W_{\epsilon} = W + \tilde{W}_{\epsilon}, \quad \tilde{W}_{\epsilon} \sim \text{Cauchy}(0, \epsilon) \text{ and independent of the system}.
\]

As $\epsilon \to 0$, the working model \eqref{eq:working-model} converges to the original causal random coefficient model \eqref{eq:causal-RCM-1}, ensuring that the estimated distribution of $R$ remains approximately unchanged. Meanwhile, the additional noise term $\tilde{W}_{\epsilon}$ enhances model flexibility.

The following proposition formalizes this idea.

\begin{proposition}\label{prop: causal_heavy_tail_2} 
Assume Assumption \ref{assumption: causal tail}, $\vert W \vert \leq M_{W}$ for constant $M_{W}$, and density $f_{W_{\epsilon}}$ exists. Then, for any given $\epsilon > 0$, the random vector $X_{\epsilon} = (Z, W_{\epsilon}, 1)$ satisfies
\[
f_{(Z, W_{\epsilon})}(T, t) \geq C_{Z} \pi^{-1}\epsilon\, \max\{1, \epsilon + M_{W}\}^{-2}(1+\Vert (T, t)\Vert_{2})^{-2\max\{2, \kappa\}}.
\]
Therefore, with $\alpha = 1/(2\max\{2, \kappa\} + 1)$, $\tau_{0} = \mathcal{S}(\mathbb{S}^{d-1})C_{Z}\epsilon / \pi \max\{1, \epsilon + M_{W}\}^{2}2^{6\max\{2, \kappa\} + 3}$, and $\rho_{0} = 1/4$, Assumption \ref{assumption: unbounded covariates} holds as established in Proposition \ref{prop:heavy-tail}.
\end{proposition}
Proofs are provided in Section \ref{subsection: proofs of causal inference} and empirical results for this causal approach are put in Section \ref{subsection: causal rcm}.

\section{Numerical experiments}\label{section: simulations}

\subsection{Simulation studies}

This section evaluates the performance of the algorithms using the folllowing numerical simulations. In the following, covariates are drawn from the von Mises-Fisher distribution with a concentration parameter $0.1$ and a mean direction $(d^{-1/2},\dots,d^{-1/2})$, satisfying Assumption \ref{assumption: bounded covariates}. We consider three types of underlying uniform distributions for the coefficient samples:
\[
\mu_{\text{sph}},~~ \mu_{\text{deg}},~~ {\rm and}~~ \mu_{\text{dis}}, 
\]
with supports of different dimensions and $R = 10$ as
\begin{gather*}
    \text{supp} \,\mu_{\text{sph}} = \Big\{\, U\in\mathbb{R}^{d}~\,\vert\, ~\Vert U - \delta R/2\, e_{2}\Vert_{2}\leq R/4,\,  \delta \in \{1, -1\}\,\Big\},\\
    \text{supp}\, \mu_{\text{deg}} =  \Big\{\, U\in\mathbb{R}^{d}~\,\vert\, ~\Vert U - \delta R/2\, e_{2}\Vert_{2}= R/4, \, \delta \in \{1, -1\}\,\Big\},\\
    \text{supp}\, \mu_{\text{dis}} = \Big\{\,U \in \mathbb{R}^{d}~\,\vert\, ~U = \delta R/2 \, e_{i},\, 1\leq i \leq d,\, \delta \in \{1, -1\}\,\Big\}. 
\end{gather*}

The simulations are performed on an 11th Gen Intel Core i7-11700 @ 2.50GHz CPU, and all reproducible codes are available at \url{https://github.com/keunwoolim/rcm_sw}. 


The performance of the algorithms is measured by computation time (in seconds) and the empirical distance (in $SW_{2}$ distance) between the original samples and the output, averaged over $100$ experiments. The empirical $SW_{2}$ distance using the Monte Carlo method is computed with the POT package \citep{flamary2021pot} with $100$ unit vector samples. 

For implementing Algorithms \ref{algorithm: bcd} and \ref{algorithm: abcd}, we perform $20$ iterations, setting $m = 50$ and $m = 1000$ respectively.  We set $k = \lceil{n^{d/(2d-1)}} \rceil$, which corresponds to the general $d\geq 6$ case in Theorem \ref{theorem: bounded estimator} and is empirically observed to be more robust than the suggested diverging rates in Theorem \ref{theorem: bounded estimator}. For Algorithm \ref{algorithm: flow}, we set $L = k$, $m = 50$, $t = 20$, $h = 1$, and $\lambda = 0.01$. The convex optimization problem for Algorithm \ref{algorithm: bcd} is solved using the CVXPY package \citep{diamond2016cvxpy}. 

Table \ref{table: simulation_dimension} reports the results of Algorithms \ref{algorithm: bcd}-\ref{algorithm: flow} for dimensions $2 \leq d \leq 5$ with $n = 500$ samples from $\mu_{\text{sph}}$. This demonstrates the efficiency of Algorithm \ref{algorithm: abcd}, as it is averagely $2.1$ times faster than Algorithm \ref{algorithm: bcd} while maintaining similar empirical accuracy. The approximation accuracy of Algorithm \ref{algorithm: flow} is lower but not substantially worse than the first two algorithms. This justifies its validity and also aligns with the purpose of generative modeling.   

\begin{table}[!ht]\renewcommand{\arraystretch}{1.1}
  \caption{Average performance of Algorithms \ref{algorithm: bcd} - \ref{algorithm: flow} with regard to the dimension}
\centering
  \begin{tabular}{ccccccc}
    \toprule
    \multirow{2}{*}{Dimension} &
      \multicolumn{2}{c}{Algorithm \ref{algorithm: bcd}} &
      \multicolumn{2}{c}{Algorithm \ref{algorithm: abcd}} & \multicolumn{2}{c}{Algorithm \ref{algorithm: flow}} \\
      & {Distance} & {Time} & {Distance} & {Time} & {Distance} & {Time} \\
      \midrule
    2 & 0.697 & 4.344 & 0.690 & 2.008 & 0.717 & 5.574\\
    3 & 0.762 & 2.935 & 0.752 & 1.366 & 0.808 & 2.580\\
    4 & 0.832 & 2.512 & 0.787 & 1.156 & 0.937 & 1.808\\
    5 & 0.943 & 2.271 & 0.910 & 1.061 & 1.057 & 1.528\\
    \bottomrule
  \end{tabular}
  \label{table: simulation_dimension}
\end{table}

Table \ref{table: simulation_sample size} reports the performance of Algorithm \ref{algorithm: abcd} for $n = 500, 1000, 1500, 2000$ samples from $\mu_{\text{sph}}$, $\mu_{\text{deg}}$, and $\mu_{\text{dis}}$, with $d = 2$. For all types of probability measures, the average empirical distance decreases as the sample size increases, while maintaining empirically robust computation time.

\begin{table}[!ht]\renewcommand{\arraystretch}{1.1}
  \caption{Average performance of Algorithm \ref{algorithm: abcd} with regard to the sample size}
\centering
  \begin{tabular}{ccccccc}
    \toprule
    \multirow{2}{*}{Sample Size} &
      \multicolumn{2}{c}{$\mu_{\text{sph}}$} &
      \multicolumn{2}{c}{$\mu_{\text{deg}}$} &
      \multicolumn{2}{c}{$\mu_{\text{dis}}$} \\
      & {Distance} & {Time} & {Distance} & {Time} & {Distance} & {Time} \\
      \midrule
    500 & 0.690 & 2.008 & 0.639 & 1.998 & 0.877 & 2.006 \\
    1000 & 0.543 & 3.162 & 0.462 & 3.218 & 0.805 & 3.170 \\
    1500 & 0.440 & 4.250 & 0.424 & 4.256 & 0.777 & 4.262 \\
    2000 & 0.375 & 4.982 & 0.373 & 5.035 & 0.725 & 5.060 \\
    \bottomrule
  \end{tabular}
  \label{table: simulation_sample size}
\end{table}

 
Lastly, Figure \ref{figure: introduction} plots the coefficient samples and outputs (i.e., $\hat w$'s introduced in \eqref{eq:points}) of Algorithms \ref{algorithm: bcd}-\ref{algorithm: flow} for $100$ experiments. For Algorithms \ref{algorithm: bcd} and \ref{algorithm: abcd}, we performed $20$ iterations with $n = 2000$ and $k = 159$, setting $m = 50$ and $m = 1000$ respectively. For Algorithm \ref{algorithm: flow}, the parameters are $L = 159$, $m = 10$, $t = 20$, $h = 1$, and $\lambda = 0.01$.

\begin{figure}[h]
    \centering\hspace{-0cm}
    \includegraphics[scale=0.6]{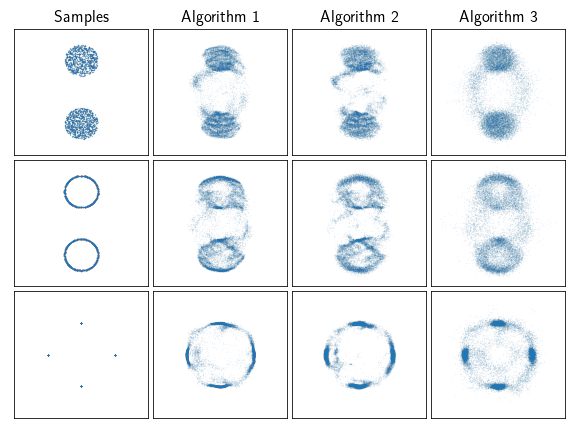}
    \caption{Coefficient samples and outputs of Algorithms \ref{algorithm: bcd}-\ref{algorithm: flow}.}
    \label{figure: introduction}
\end{figure}

Figure \ref{figure: generation} plots the outputs of $100$ experiments of Algorithm \ref{algorithm: flow} depending on $L$ and $\lambda$, with fixed parameters $n = 2000$, $m = 10$, $t = 20$, and $h = 1$. In Figure \ref{figure: generation_particles}, $\lambda$ is fixed as $0.01$, and $L$ varies as $10$, $20$, $40$, and $80$. Meanwhile, in Figure \ref{figure: generation_regularization}, $L$ is fixed as $20$, and $\lambda$ varies as $0.01$, $0.02$, $0.04$, and $0.08$. Intuitively, the particle systems approximation becomes more precise as $L$ increases, and the generated samples become more diverse as $\lambda$ increases.

\begin{figure}\hspace{-0.5cm}
     \begin{subfigure}[h]{0.5\textwidth}
         \centering
         \includegraphics[width=\linewidth]{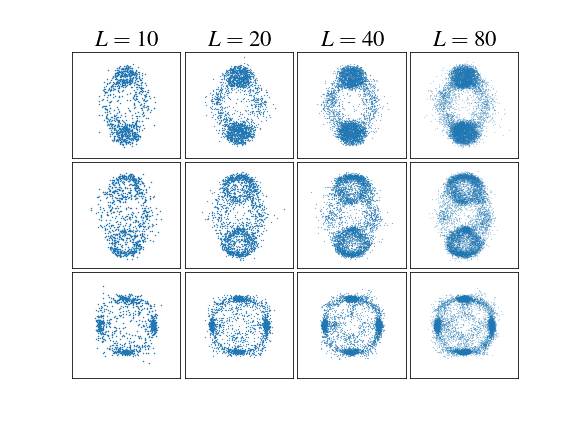}
         \caption{Varying number $L$ of particles.}
         \label{figure: generation_particles}
     \end{subfigure}
     \hfill
     \begin{subfigure}[h]{0.5\textwidth}
         \centering
         \includegraphics[width=\linewidth]{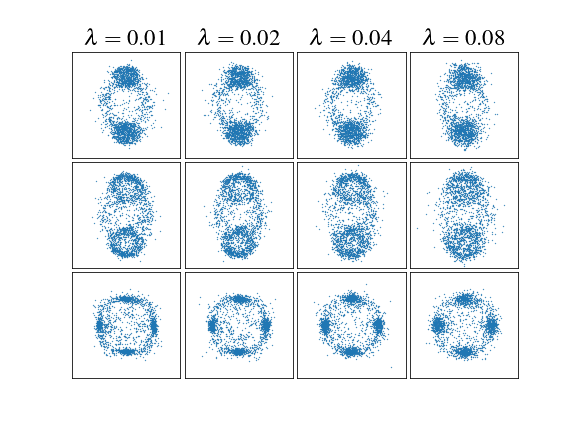}
         \caption{Varying regularization parameter $\lambda$.}
         \label{figure: generation_regularization}
     \end{subfigure}
        \caption{Outputs of Algorithm \ref{algorithm: flow} with varying numbers of particles and regularization parameter.}
        \label{figure: generation}
\end{figure}

\subsection{Causal RCM}\label{subsection: causal rcm}

Based on the regularized working model \eqref{eq:working-model} for causal inference, we analyze the AIDS Clinical Trials Group Study 175 (ACTG 175) data provided in the R \texttt{speff2trial} package \citep{speff2trial} to evaluate the treatment effects of therapy for human immunodeficiency virus type I. Specifically, using the monotherapy of zidovudine (Therapy 1) as the control group, we examine the effects of combination therapies—zidovudine and didanosine (Therapy 2) and zidovudine and zalcitabine (Therapy 3)—as well as monotherapy with didanosine (Therapy 4), as considered in \citet[Table 2]{hammer1996trial}. 

We analyze each investigational treatment (Therapies 2-4) separately against the control. In each analysis, the random treatment variable $W$ is assigned as $W = 0$ for the control group and $W = 1$ for the treatment group. The pretreatment covariates $Z$ include the baseline CD4 T cell count and the number of days of prior antiretroviral therapy, analogous to the setting in \citet[Section 6]{10.1093/biomet/asad045}. After standard normalization, this design produces three sets of covariate samples, each with $d = 4$ and respective sample sizes of $n = 1054, 1056, 1093$ for Therapies 2–4. The outcome variable $Y$ is the CD4 T cell count measured at $20 \pm 5$ weeks, which is also normalized.

We implemented a modified version of Algorithm \ref{algorithm: abcd} and conducted 100 experiments for each of the three cases. The parameter $k = \lceil{n^{d/(2d-1)}} \rceil$ was computed as $k = 54, 54, 55$, respectively, and additional parameters were set as $\epsilon = 0.005$, $R = 10$, and $m = 1000$, with 20 iterations.  

\begin{figure}[h]
    \centering
    \includegraphics[scale=0.72]{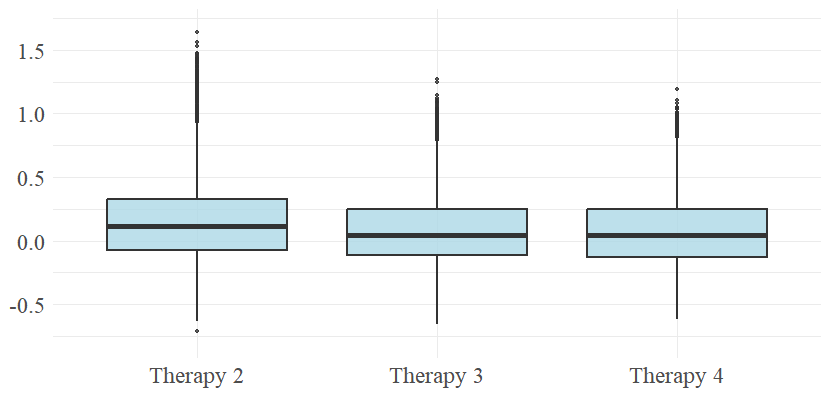}
    \caption{Treatment coefficient samples for Therapies 2–4.}
    \label{figure: causal}
\end{figure}

Figure \ref{figure: causal} presents the box plots for the estimated treatment coefficients corresponding to Therapies 2–4. The estimation results also show the estimated percentiles of the unit treatment effect. For example, the estimated median treatment effects for Therapies 2-4 are $0.113$, $0.047$, and $0.045$, respectively. These results suggest that all three therapies are effective compared to zidovudine monotherapy, with the combination of zidovudine and didanosine showing particularly notable effectiveness. These findings are generally consistent with the results of \cite{hammer1996trial}.

\appendix

\section{Proofs}\label{section: proofs}

\subsection{Proofs of Assumptions}
\begin{proof}[Proof of Proposition \ref{prop:heavy-tail}:] Introduce an associated random variable $\widebar{X} := (-1)^{\delta}\tilde{X}$, where $\delta \sim {\rm Bern}(1/2)$ is independent of $X$. Then, we prove that the density function $f_{\widebar{X}}$ of $\widebar{X}$ exists, and for all $0<\alpha\leq 1/(\kappa+1)$, there exists some constant $C_{\widebar{X}}>0$ such that 
\begin{align*}
    D(\tau) = \{\,\widebar{X}\in \mathbb{S}^{d-1}\, \vert \, \vert \langle \widebar{X}, e_{1}\rangle\vert\leq C_{\widebar{X}}\tau^{\alpha}\, \}\,\,\text{ and }\,\,\sigma(\{V\in\mathbb{S}^{d-1}\vert f_{\widebar{X}}(V)\leq \tau\} \backslash D(\tau)) = 0
\end{align*}
hold for all $0<\tau\leq \tau_{0}$. In particular, we give the proof with the parameters $\alpha = 1/(\kappa + 1)$, $\tau_{0} = \mathcal{S}(\mathbb{S}^{d-1})C_{f}/2^{3\kappa + 3}$, $\rho_{0} = 1/4$, and $C_{\widebar{X}} = 4/(\mathcal{S}(\mathbb{S}^{d-1})C_{f})^{\alpha}$.
    
\par\noindent\textbf{Step 1.} Denote the sets $\tilde{D}$ and $\tilde{D}(\tau)$ with the vector $\tilde{T}\in\mathbb{S}^{d-1}$ as
\begin{align*}
    \tilde{D} = \{\,\tilde{T}\in \mathbb{S}^{d-1}\, \vert \, \langle \tilde{T}, e_{1}\rangle> 0\, \} \,\,\text{ and }\,\, \tilde{D}(\tau) = \{\,\tilde{T}\in \mathbb{S}^{d-1}\, \vert \, \langle \tilde{T}, e_{1}\rangle> C_{\widebar{X}}\tau^{\alpha}\, \}
\end{align*}
and let the vector $\tilde{T} = (\tilde{T}_{1}, \tilde{T}_{2:d})$ with $\tilde{T}_{2:d}\in \mathbb{R}^{d-1}$. In particular, there exists one to one correspondence between $\tilde{T}\in\tilde{D}$ and $\tilde{T}_{2:d}\in\tilde{A}$ for
\begin{align*}
    \tilde{A} = \{\,\tilde{T}_{2:d}\in \mathbb{R}^{d-1}\,\vert \,\Vert \tilde{T}_{2:d}\Vert_{2}<1\,\}. 
\end{align*}
Then, define the injective function $H\colon \mathbb{R}^{d-1}\to \tilde{A}$ and its inverse $H^{-1}\colon \tilde{A}\to \mathbb{R}^{d-1}$ as
    \begin{align*}
        \tilde{T}_{2:d} = H(T_{2:d}) = T_{2:d} / (1+ \Vert T_{2:d}\Vert_{2}^{2})^{1/2}\,\,\, \text{ and } \,\,\, T_{2:d} = H^{-1}(\tilde{T}_{2:d}) = \tilde{T}_{2:d} / (1- \Vert \tilde{T}_{2:d}\Vert_{2}^{2})^{1/2}, 
    \end{align*}
where we denote $T = (1, T_{2:d})$ on the domain of $X$. Without loss of generality, we focus on variables with restricted domain as $X_{2:d}\in\mathbb{R}^{d-1}$, $\widebar{X}_{2:d}\in\tilde{A}$, and $\delta = 0$.

\par\noindent\textbf{Step 2.} Denote the Jacobian of $H^{-1}$ as $J_{H^{-1}}$ and $I$ as the identity matrix. We have that the determinant of $J_{H^{-1}}$ at $\tilde{T}_{2:d}$ is lower bounded as
    \begin{align*}
        \det J_{H^{-1}}(\tilde{T}_{2:d}) = \det((1- \Vert \tilde{T}_{2:d}\Vert_{2}^{2})^{-\frac{3}{2}}\tilde{T}_{2:d}\,\tilde{T}_{2:d}^{\top}+(1- \Vert \tilde{T}_{2:d}\Vert_{2}^{2})^{-\frac{1}{2}}I)\geq (1- \Vert \tilde{T}_{2:d}\Vert_{2}^{2})^{-\frac{d-1}{2}}\geq 1,
    \end{align*}
    by Minkowski's determinant theorem as in \cite{marcus1971extension}. 
    Since $\widebar{X}_{2:d} = H(X_{2:d})$, $\widebar{X}_{2:d}$ is a random variable on $\tilde{A}$ with density function $f_{\widebar{X}_{2:d}}$ such that 
    \begin{align*}
        f_{\widebar{X}_{2:d}}(\tilde{T}_{2:d}) = f_{X_{2:d}}(T_{2:d})\vert \det J_{H^{-1}}(\tilde{T}_{2:d})\vert / 2 \geq f_{X_{2:d}}(T_{2:d}) / 2\geq C_{f}/2 (1+ \Vert T_{2:d} \Vert_{2})^{\kappa}.  
    \end{align*} 
    \par\noindent\textbf{Step 3.} Now, we derive the lower bound of the density function $f_{\widebar{X}}\vert_{\tilde{D}(\tau)}$ restricted on the set $\tilde{D}(\tau)$, where the existence of $f_{\widebar{X}}$ is directly obtained by the Radon-Nikodym theorem. Consider the Borel set $\tilde{K}\subset\tilde{D}(\tau)$ with area $\mathcal{S}(\tilde{K})$ and corresponding set $K\subset \mathbb{R}^{d}$ such that
    \begin{align*}
        K = \Big\{\,(0, \tilde{T}_{2:d})\in \mathbb{R}^{d}\,\vert\, \tilde{T} = (\tilde{T}_{1}, \tilde{T}_{2:d})\in \tilde{K}\,\Big\}
    \end{align*}
    with area $\mathcal{S}(K)$, and let the set $\tilde{A}(\tau)$ be
    \begin{align*}
        \tilde{A}(\tau) = \{\,\tilde{T}_{2:d}\in \mathbb{R}^{d-1}\,\vert\, \Vert \tilde{T}_{2:d}\Vert_{2}<(1 - C_{\widebar{X}}^{2}\tau^{2\alpha})^{1/2}\,\}.
    \end{align*}
    Then, denoting the function $l\colon \tilde{A}(\tau)\to \mathbb{R}$ as
    \begin{align*}
        l(\tilde{T}_{2:d}) = (1 - \Vert \tilde{T}_{2:d}\Vert_{2}^{2})^{1/2}, 
    \end{align*}
    by direct computation we have 
    \begin{align*}
        (1+ \Vert \nabla l(\tilde{T}_{2:d})\Vert_{2}^{2})^{1/2} = (1 - \Vert \tilde{T}_{2:d}\Vert_{2}^{2})^{-1/2}\leq C_{\widebar{X}}^{-1}\tau^{-\alpha},
    \end{align*}
    which leads to the inequality
    \begin{align*}
        \int_{\tilde{K}}& f_{\widebar{X}}\vert_{\tilde{D}(\tau)}\, \dd \sigma  = \int_{K} f_{\widebar{X}_{2:d}}\, \dd S \geq \frac{C_{f}}{2}\big(\,1+\frac{(1 - C_{\widebar{X}}^{2}\tau^{2\alpha})^{1/2}}{C_{\widebar{X}}\tau^{\alpha}}\,\big)^{-\kappa} \mathcal{S}(K)\\
        &\geq \frac{C_{f}C_{\widebar{X}}\tau^{\alpha}}{2}\big(\,1+\frac{(1 - C_{\widebar{X}}^{2}\tau^{2\alpha})^{1/2}}{C_{\widebar{X}}\tau^{\alpha}}\,\big)^{-\kappa} \mathcal{S}(\tilde{K}) = \mathcal{S}(\mathbb{S}^{d-1}) \int_{\tilde{K}}\frac{C_{f}C_{\widebar{X}}\tau^{\alpha}}{2}\big(\,1+\frac{(1 - C_{\widebar{X}}^{2}\tau^{2\alpha})^{1/2}}{C_{\widebar{X}}\tau^{\alpha}}\,\big)^{-\kappa}\dd \sigma.
    \end{align*}
    Since this inequality holds for any Borel set $\tilde{K}\subset \tilde{D}(\tau)$, applying the Radon-Nikodym Theorem again, we obtain that
    \begin{align*}
        f_{\widebar{X}}\vert_{\tilde{D}(\tau)}\geq \frac{\mathcal{S}(\mathbb{S}^{d-1})C_{f} C_{\widebar{X}}\tau^{\alpha}}{2}\big(\,1+\frac{(1 - C_{\widebar{X}}^{2}\tau^{2\alpha})^{1/2}}{C_{\widebar{X}}\tau^{\alpha}}\,\big)^{-\kappa}
    \end{align*}
    holds almost everywhere. Finally, constant $C_{\widebar{X}} = 4/(\mathcal{S}(\mathbb{S}^{d-1})C_{f})^{\alpha}$ satisfies the inequality
    \begin{align*}
        \frac{\mathcal{S}(\mathbb{S}^{d-1})C_{f} C_{\widebar{X}}\tau^{\alpha}}{2}\Big(\,1+\frac{(1 - C_{\widebar{X}}^{2}\tau^{2\alpha})^{1/2}}{C_{\widebar{X}}\tau^{\alpha}}\,\Big)^{-\kappa}&\geq \frac{\mathcal{S}(\mathbb{S}^{d-1})C_{f} C_{\widebar{X}}\tau^{\alpha}}{2}(1+C_{\widebar{X}}^{-1}\tau^{-\alpha})^{-\kappa}\\ 
        &\geq 2^{-\kappa}\mathcal{S}(\mathbb{S}^{d-1})C_{f} C_{\widebar{X}}\tau^{\alpha}(1+ C_{\widebar{X}}^{-\kappa}\tau^{-\alpha\kappa})^{-1}>\tau
    \end{align*}
    for all $0<\tau\leq \tau_{0}$, and the proof is thus complete. 
\end{proof}

\subsection{Proofs of results concerning estimation}\label{subsection: proofs of estimation}

\begin{proof}[Proof of Theorem \ref{proposition: convex}:]
Recall that
    \begin{align*}
        \mathcal{F}_{k}\circ \eta(w) = \int_{\mathbb{S}^{d-1}} W_{2}^{2}(\widebar{\mu}_{k}^{V}, \eta(w)^{V})\dd \sigma(V),
    \end{align*}
with $w = (w_{1},\dots,w_{N})$ and $w_{i} = (w_{i,1},\dots,w_{i,d})$. For two points $w, \tilde{w}\in \widebar{B}_{R}(0)^{N}$ with $\tilde{w} =(\tilde{w}_{1},\dots,\tilde{w}_{N})$, by the triangle inequality
    \begin{align*}
        \vert \mathcal{F}_{k}\circ &\eta(w) - \mathcal{F}_{k}\circ \eta(\tilde{w})\vert \leq \int_{\mathbb{S}^{d-1}} \vert W_{2}^{2}(\widebar{\mu}_{k}^{V}, \eta(w)^{V})-W_{2}^{2}(\widebar{\mu}_{k}^{V}, \eta(\tilde{w})^{V})\vert \dd \sigma(V)\\
        &\leq 4R \int_{\mathbb{S}^{d-1}} \vert W_{2}(\widebar{\mu}_{k}^{V}, \eta(w)^{V})-W_{2}(\widebar{\mu}_{k}^{V}, \eta(\tilde{w})^{V})\vert \dd \sigma(V)\leq 4R \int_{\mathbb{S}^{d-1}} W_{2}(\eta(w)^{V}, \eta(\tilde{w})^{V})\dd \sigma(V).
    \end{align*}
Additionally, we have
    \begin{align*}
        W_{2}^{2}(\eta(w)^{V}, \eta(\tilde{w})^{V})\leq N^{-1}\sum_{i = 1}^{N}(\langle w_{i}, V \rangle - \langle\tilde{w}_{i}, V \rangle)^{2} \leq N^{-1} \sum_{i = 1}^{N}\Vert w_{i} - \tilde{w}_{i}\Vert^{2}= N^{-1}\Vert w - \tilde{w}\Vert^{2}
    \end{align*}
    holds by definition of the $W_{2}$ distance, which yields that $\mathcal{F}_{k}\circ \eta$ is $4RN^{-1/2}$-Lipschitz continuous. Since the function $\mathcal{F}_{k}\circ \eta$ is continuous on the compact domain $\widebar{B}_{R}(0)^{N}$, there then exists global minimizer[s] \cite[Theorem 4.15]{rudin1964principles}. In the case of $N = k = n$, $\hat{\mu}_{\beta} = n^{-1}\sum_{i = 1}^{n}\delta_{\tilde{Y}_{i}\tilde{X}_{i}}$ is the unique global minimizer of \eqref{eq:mde}.
\end{proof}

\begin{proof}[Proof of Theorem \ref{theorem: bounded estimator}:]

By Theorem \ref{proposition: convex}, there exist global minimizer[s] of $\mathcal{F}_{k}\circ \eta$, and the existence of the estimator $\hat{\mu}_{\beta}$ is directly obtained. In the following steps, we prove the convergence of $SW_{2}$ distance based on the convergence of each $W_{2}$ distance between projected measures. 

\par\noindent\textbf{Step 1.} First, we analyze the bound of $SW_{2}$ distance with the minimum distance property. Introduce $N$ independent copies of $\beta$ as $\tilde{\beta}_{1},\dots,\tilde{\beta}_{N}$ and the corresponding empirical measure $\mu_{N}\in \mathcal{P}_{N}(\widebar{B}_{R}(0))$ as $\mu_{N} = N^{-1}\sum_{i = 1}^{N}\delta_{\tilde{\beta}_{i}}$. 

Then, by the minimum distance property of $\hat{\mu}_{\beta}$ on $\mathcal{P}_{N}(\widebar{B}_{R}(0))$, we have
\begin{align*}
    \int_{\mathbb{S}^{d-1}} W_{2}^{2}(\widebar{\mu}_{k}^{V}, \hat{\mu}_{\beta}^{V})\dd \sigma(V) \leq \int_{\mathbb{S}^{d-1}} W_{2}^{2}(\widebar{\mu}_{k}^{V}, \mu_{N}^{V})\dd \sigma(V),
\end{align*}
which leads to the bound of the $SW_{2}$ distance
\begin{align}
    SW_{2}^{2}(\mu_{\beta}, \hat{\mu}_{\beta}) &= \int_{\mathbb{S}^{d-1}} W_{2}^{2}(\mu_{\beta}^{V}, \hat{\mu}_{\beta}^{V})\dd \sigma(V)\notag\\
    &\leq 2\int_{\mathbb{S}^{d-1}}W_{2}^{2}(\mu_{\beta}^{V}, \widebar{\mu}_{k}^{V})\dd \sigma(V) +  2\int_{\mathbb{S}^{d-1}} W_{2}^{2}(\widebar{\mu}_{k}^{V}, \hat{\mu}_{\beta}^{V})\dd \sigma(V)\nonumber\\
    &\leq 2\int_{\mathbb{S}^{d-1}}W_{2}^{2}(\mu_{\beta}^{V}, \widebar{\mu}_{k}^{V})\dd \sigma(V) +  2\int_{\mathbb{S}^{d-1}} W_{2}^{2}(\widebar{\mu}_{k}^{V}, \mu_{N}^{V})\dd \sigma(V)\nonumber\\
    &\leq 6 \int_{\mathbb{S}^{d-1}}W_{2}^{2}(\mu_{\beta}^{V}, \widebar{\mu}_{k}^{V})\dd \sigma(V) +  4\int_{\mathbb{S}^{d-1}} W_{2}^{2}(\mu_{\beta}^{V}, \mu_{N}^{V})\dd \sigma(V).\label{eq: SW bound}
\end{align}
Regarding the second term in \eqref{eq: SW bound}, since $\mu_{N}^{V}$ corresponds to the empirical measure of $\mu_{\beta}^{V}$, \citet[Theorem 1]{Fournier2015} directly implies the existence of some constant $C = C(R, d)$ such that
    \begin{align*}
    \begin{cases}
        \mathbb{E} W_{2}^{2}(\mu_{\beta}^{V}, \mu_{N}^{V})\leq C N^{-1/3} &\text{ for } d\leq 5,\\
        \mathbb{E} W_{2}^{2}(\mu_{\beta}^{V}, \mu_{N}^{V})  \leq C N^{-2/d}&\text{ for } d\geq 6.
    \end{cases}
    \end{align*}

\par\noindent\textbf{Step 2.} We then derive the bound of the first term in \eqref{eq: SW bound}. To this end, let's introduce an oracle NN-induced random measure as
\begin{align*}
    \tilde{\mu}^{V}_{k} := \frac{1}{k} \sum_{i\in S_{X}(V,k)} \delta_{P^{V}(\beta_{i})}.
\end{align*}
The first term in \eqref{eq: SW bound} can then be upper bounded as
\begin{align}\label{eq: SW bound_1}
    \int_{\mathbb{S}^{d-1}}\mathbb{E}W_{2}^{2}(\mu_{\beta}^{V}, \widebar{\mu}_{k}^{V})\dd \sigma(V) \leq 2\int_{\mathbb{S}^{d-1}}\mathbb{E} W_{2}^{2}(\mu_{\beta}^{V}, \tilde{\mu}_{k}^{V})\dd \sigma (V)+2\int_{\mathbb{S}^{d-1}}\mathbb{E}W_{2}^{2}(\tilde{\mu}_{k}^{V}, \widebar{\mu}_{k}^{V})\dd \sigma(V). 
\end{align}
For the first term in \eqref{eq: SW bound_1}, by independence between $\beta$ and $X$, 
\citet[Theorem 1]{Fournier2015} can still be used to show that
\begin{align*}
    \begin{cases}
        \mathbb{E}W_{2}^{2}(\mu_{\beta}^{V},\tilde{\mu}_{k}^{V}) \leq C k^{-1/3}, &\text{ for } d\leq 5,\\
        \mathbb{E}W_{2}^{2}(\mu_{\beta}^{V},\tilde{\mu}_{k}^{V}) \leq C k^{-2/d},&\text{ for } d\geq 6
    \end{cases}
    \end{align*}
holds for some constant $C = C(R, d)$.

It remains to upper bound the second term in \eqref{eq: SW bound_1}. To this end, let's denote the set $S(V, \gamma)$ and its range bound $\gamma(V, k)$ for $V\in\mathbb{S}^{d-1}$ and $0 \leq \gamma\leq \sqrt{2}$ as 
\begin{align*}
    S(V, \gamma) &= \{\,\tilde{X}\in\mathbb{S}^{d-1}\,\vert\, \langle V, \tilde{X}\rangle \geq 1 - \gamma^{2}\,\}\\
{\rm and}~~~    \gamma(V, k) &= \min\{\,\gamma\colon \{\tilde{X}_{i}\,\vert\, i \in S_{X}(V, k)\} \subset S(V, \gamma)\,\}.
\end{align*}
By Lemma \ref{lemma: vector}, we obtain an upper bound of $P^{V}(\beta_{i})$ as
\begin{align*}
    \vert P^{V}(\beta_{i}) - P^{V}(\tilde{Y}_{i}\tilde{X}_{i})\vert =  \vert \langle \beta_{i}, V\rangle - \langle \beta_{i}, \tilde{X}_{i}\rangle \langle \tilde{X}_{i}, V\rangle \vert \leq C \gamma(V, k)
\end{align*}
with constant $C = C(R)>0$, which leads to the $W_{2}$ bound
\begin{align*}
    W_{2}(\tilde{\mu}_{k}^{V}, \widebar{\mu}_{k}^{V})\leq C\gamma(V, k) 
\end{align*}
by the definition of $W_{2}$ distance, uniformly on $V$ and $k$. 

\par\noindent\textbf{Step 3, Case I.} Under Assumption \ref{assumption: bounded covariates}, by Lemma \ref{lemma: expectation},     
\begin{align*}
    \int_{\mathbb{S}^{d-1}}\mathbb{E}W_{2}^{2}(&\tilde{\mu}_{k}^{V}, \widebar{\mu}_{k}^{V})\dd \sigma(V)\leq C(k/n)^{\frac{2}{d-1}} 
\end{align*}
 holds with some constant $C = C(R, d,\tau_{0})$, where $\ind(\cdot)$ is the indicator function. Combining all the results and letting $N = k = n^{6/(d+5)}$ for $2\leq d \leq 5$ and $N = k = n^{d/(2d-1)}$ for $d \geq 6$, we have the $SW_{2}$ bound 
\begin{align*}
\begin{cases}
    \mathbb{E}[SW_{2}^{2}(\mu_{\beta}, \hat{\mu}_{\beta})] \leq C n^{-2/(d+5)} &\text{ for } d\leq 5,\\
    \mathbb{E}[SW_{2}^{2}(\mu_{\beta}, \hat{\mu}_{\hat{\beta}})] \leq C n^{-2/(2d-1)}&\text{ for }d\geq 6.
\end{cases}
\end{align*}
By Jensen's inequality, 
\begin{align*}
    \mathbb{E}[SW_{2}(\mu_{\beta}, \hat{\mu}_{\beta})]\leq \mathbb{E}[SW_{2}^{2}(\mu_{\beta}, \hat{\mu}_{\beta})]^{\frac{1}{2}}
\end{align*}
holds, and the proof of the first case is then complete. 

\vspace{0.2cm}
\par\noindent\textbf{Step 3, Case II.} Under Assumption \ref{assumption: unbounded covariates}, by Lemma \ref{lemma: expectation}, 
\begin{align*}
    \int_{\mathbb{S}^{d-1}}\mathbb{E}W_{2}^{2}&(\tilde{\mu}_{k}^{V}, \widebar{\mu}_{k}^{V})\dd \sigma(V)\leq C((k/n\tau)^{\frac{2}{d-1}}+\tau^{\alpha}+\rho),
\end{align*}
with some constant $C = C(R, d, \tau_{0}, \rho_{0}, C_{\tilde{X}})$. Lastly, choosing
\begin{align*}
    N = k = n^{\frac{6\alpha}{d\alpha+5\alpha+2}},\,\,\, \tau = \tau_{0}n^{-\frac{2}{d\alpha+5\alpha+2}},\,\,\,\text{and}\,\,\,\rho = \rho_{0}n^{-\frac{\alpha}{d\alpha+5\alpha+2}}\log n,
\end{align*}
we obtain the bound for $2 \leq d \leq 5$ by direct computation
\begin{align*}
    \mathbb{E}[SW_{2}^{2}(\mu_{\beta}, \hat{\mu}_{\beta})] \leq C n^{-\frac{\alpha}{d\alpha+5\alpha+2}}\log n,
\end{align*}
with some constant $C = C(R, d,\tau_{0},\rho_{0}, C_{\tilde{X}})$. In the final case of $d \geq 6$, choosing
\begin{align*}
    N = k = n^{\frac{d\alpha}{2d\alpha-\alpha+2}},\,\,\, \tau = \tau_{0}n^{-\frac{2}{2d\alpha-\alpha+2}},\,\,\, \text{and} \,\,\,\rho = \rho_{0}n^{-\frac{\alpha}{2d\alpha-\alpha+2}}\log n,
\end{align*}
we obtain the bound
\begin{align*}
    \mathbb{E}[SW_{2}^{2}(\mu_{\beta}, \hat{\mu}_{\beta})] \leq C n^{-\frac{\alpha}{2d\alpha-\alpha+2}}\log n
\end{align*}
with constant $C = C(R, d,\tau_{0}, \rho_{0}, C_{\tilde{X}})$. This completes the proof.
\end{proof}

\subsection{Proofs of results concerning computation}\label{subsection: proofs of computation}

\begin{proof}[Proof of Proposition \ref{proposition: discrete continuity}:]
We begin with a proof of the Lipschitz continuity of the function $W_{2}^{2}(\widebar{\mu}_{k}^{V}, \eta(w)^{V})$ in $w$. For any vector $V\in\mathbb{S}^{d-1}$ and two points $w = (w_{1},\dots, w_{k})$ and $\tilde{w} = (\tilde{w}_{1},\dots,\tilde{w}_{k})$ on $\widebar{B}_{R}(0)^{N}$,  we have
\begin{align}\label{eq: lipschitz wasserstein}
    \Big\vert W_{2}^{2}(\widebar{\mu}_{k}^{V}, \eta(w)^{V}) - W_{2}^{2}(\widebar{\mu}_{k}^{V}, \eta(\tilde{w})^{V})\Big\vert &\leq 4R \,\Big\vert W_{2}(\widebar{\mu}_{k}^{V}, \eta(w)^{V}) - W_{2}(\widebar{\mu}_{k}^{V}, \eta(\tilde{w})^{V}) \Big\vert \notag\\
    &\leq 4R\, W_{2}(\eta(w)^{V}, \eta(\tilde{w})^{V}).
    \end{align}
In addition, by definition of the $W_{2}$ distance,
\begin{align*}
     W_{2}^{2}(\eta(w)^{V}, \eta(\tilde{w})^{V})\leq \frac{1}{N}\sum_{i = 1}^{N}\langle w_{i} - \tilde{w}_{i}, V\rangle^{2} \leq \frac{1}{N}\sum_{i = 1}^{N}\Vert w_{i} - \tilde{w}_{i}\Vert_{2}^{2} = N^{-1}\Vert w- \tilde{w}\Vert_{2}^{2}, 
\end{align*}
which leads to the $4RN^{-1/2}$-Lipschitz continuity of the function $W_{2}^{2}(\widebar{\mu}_{k}^{V}, \eta(w)^{V})$. 

Then, the $4RN^{-1/2}$-Lipschitz continuity of the function $\widebar{\mathcal{F}}_{k}(\eta(w))$ is directly derived by its definition as again, for any two vectors $w$ and $\tilde{w}$,
\begin{align*}
    \Big\vert \widebar{\mathcal{F}}_{k}(\eta(w)) - \widebar{\mathcal{F}}_{k}(\eta(\tilde{w})) \Big\vert \leq \frac{1}{m}\sum_{i = 1}^{m}\Big\vert W_{2}^{2}(\widebar{\mu}_{k}^{V_{i}}, \eta(w)^{V_{i}}) - W_{2}^{2}(\widebar{\mu}_{k}^{V_{i}}, \eta(\tilde{w})^{V_{i}})\Big\vert \leq 4RN^{-1/2}\Vert w - \tilde{w} \Vert_{2} 
\end{align*}
holds. Since the function $\widebar{\mathcal{F}}_{k}(\eta(w))$ is continuous on the compact domain $\widebar{B}_{R}(0)^{N}$, there exists a minimizer. For $N = k = n$, $n^{-1}\sum_{i = 1}^{n}\delta_{\tilde{Y}_{i}\tilde{X}_{i}}$ is the unique global minimizer of \eqref{eq:approximate-mde}.
\end{proof}

\begin{proof}[Proof of Proposition \ref{proposition: uniform convergence}:]

This proof is a generalization of the proof of \citet[Theorems 2.3 and 2.4]{tanguy2024properties}.   

\par\noindent\textbf{Step 1.} Regarding the first statement, we first establish results on a countable dense subset $H$ of the domain $\widebar{B}_{R}(0)^{N}$, and then extend these analyses to the entire domain. Denote $H$ as the set of vectors in $\widebar{B}_{R}(0)^{N}$ with elements in the set of rational numbers $\mathbb{Q}$:
\begin{align}\label{eq: rational number set}
    H = \big\{\,w \in \widebar{B}_{R}(0)^{N}\,\vert\,w_{i, j}\in \mathbb{Q}, 1 \leq i\leq N, 1\leq j \leq d \,\big\},     
\end{align}
which is a countable dense subset of $\widebar{B}_{R}(0)^{N}$. Then, for arbitrary fixed $w \in \widebar{B}_{R}(0)^{N}$, we have almost sure convergence
\begin{align*}
    \mathbb{P}_{\sigma}\,\big(\,\lim_{m\to \infty}\widebar{\mathcal{F}}_{k}(\eta(w)) = \mathcal{F}_{k}(\eta(w)\, \big) = 1, 
\end{align*}
by the strong law of large numbers. Since the set $H$ is countable, this directly leads to the almost sure convergence on the set $H$ uniformly as
\begin{align*}
    \mathbb{P}_{\sigma}\,\big(\,\lim_{m\to \infty}\widebar{\mathcal{F}}_{k}(\eta(w)) = \mathcal{F}_{k}(\eta(w)),\,\,\forall w \in H\, \big) = 1. 
\end{align*}  

\par\noindent\textbf{Step 2.} 
Since the function $\widebar{\mathcal{F}}_{k}(\eta(w))$ is $4RN^{-1/2}$-Lipschitz continuous by Proposition \ref{proposition: discrete continuity}, the sequence $(\widebar{\mathcal{F}}_{k}(\eta(w)))_{m\in\mathbb{N}}$ is uniformly equicontinuous. Then, by \citet[Proposition 3.2]{hirsch2012elements}, since the equicontinuous sequence $(\widebar{\mathcal{F}}_{k}(\eta(w)))_{m\in\mathbb{N}}$ converges to $\mathcal{F}_{k}(\eta(w))$ on the dense set $H$, the sequence converges uniformly to the continuous function $\mathcal{F}_{k}(\eta(w))$. It leads to the final convergence result
\begin{align*}
    \mathbb{P}_{\sigma}\big(\,\lim_{m\to\infty}\Vert \widebar{\mathcal{F}}_{k}(\eta(w)) - \mathcal{F}_{k}(\eta(w)) \Vert_{\infty} = 0\,\big) = 1
\end{align*}
and the proof of the first statement is complete. 
\par\noindent\textbf{Step 3.} For the second statement, we derive the $\sigma$-Donsker class by the Lipschitz continuity \cite[Example 19.7]{van2000asymptotic}, then prove the main result with the continuous mapping theorem \cite[Theorem 18.11]{van2000asymptotic}. Denote random variables $\sigma_{m} f$ and $\sigma f$ for the function $f\colon \mathbb{S}^{d-1}\to \mathbb{R}$ as
\begin{align*}
    \sigma_{m} f := \frac{1}{m}\sum_{j = 1}^{m}f(V_{i}) \,\,\, \text{ and }\,\,\, \sigma f := \int_{\mathbb{S}^{d-1}}f(V) \dd \sigma(V),
\end{align*}
where $V_{1},\dots,V_{m}$ are independently sampled from the Haar measure $\sigma$. In addition, consider the random variable $\mathbb{G}_{m}f$ and set $\mathbb{G}_{m}$ for the function class $\mathcal{T}$ as  
\begin{align*}
    \mathbb{G}_{m} = \{\,\mathbb{G}_{m}f\, \vert\, f\in \mathcal{T} \,\} \,\,\,\text{ with }\,\,\, \mathbb{G}_{m}f = \sqrt{m}(\sigma_{m} f - \sigma f). 
\end{align*}
Let $\mathcal{T}$ be a class of functions taking unit vector variables and return projected Wasserstein distances, such that 
\begin{align*}
    \mathcal{T} = \Big\{\, f_{w}\,\vert\, f_{w}\colon \mathbb{S}^{d-1}\to \mathbb{R},\, w\in \widebar{B}_{R}(0)^{N},\, f_{w}(V) =  W_{2}^{2}(\widebar{\mu}_{k}^{V}, \eta(w)^{V})\,\Big\}.  
\end{align*}
For two arbitrary vectors $w, \tilde{w}\in \widebar{B}_{R}(0)^{N}$, we have the relation between $f_{w}$ and $f_{\tilde{w}}$ as
\begin{align*}
    \vert f_{w}(V) - f_{\tilde{w}}(V)\vert 
    \leq 4R \,\vert  W_{2}(\widebar{\mu}_{k}^{V}, \eta(w)^{V}) - W_{2}(\widebar{\mu}_{k}^{V}, \eta(\tilde{w})^{V}) \vert
    \leq 4RN^{-1/2} \Vert w- \tilde{w}\Vert_{2}
\end{align*}
identical to the derivation of \eqref{eq: lipschitz wasserstein}. Then, $\mathcal{T}$ satisfies the conditions of \citet[Example 19.7]{van2000asymptotic}, which yields that $\mathcal{T}$ is $\sigma$-Donsker and the process $\mathbb{G}_{m}$ converges in distribution to the $\sigma$-Brownian bridge $\mathbb{G}_{\sigma}$ in the space $\ell^{\infty}(\mathcal{T})$ equipped with the supremum norm.

\par\noindent\textbf{Step 4.} Denote the space of bounded functions on $\widebar{B}_{R}(0)^{N}$ as $\ell^{\infty}(\widebar{B}_{R}(0)^{N})$ with supremum norm $\Vert \cdot \Vert_{\ell^{\infty}(\widebar{B}_{R}(0)^{N})}$, and let the linear function $\Psi\colon \ell^{\infty}(\mathcal{T}) \to \ell^{\infty}(\widebar{B}_{R}(0)^{N})$ as
\begin{align*}
    \Psi(s)(w) = s(f_{w})\,\,\,\text{ with }\,\,\, s \in \ell^{\infty}(\mathcal{T}),\, f_{w}\in \ell^{\infty}(\widebar{B}_{R}(0)^{N}), 
\end{align*}
which preserves the norm structure from each space and is continuous. Then, by the continuous mapping theorem, viewing the process $\sqrt{m}(\widebar{\mathcal{F}}_{k}(\eta(w)) - \mathcal{F}_{k}(\eta(w))$ as the function in $\ell^{\infty}(\widebar{B}_{R}(0)^{N})$, it converges in distribution to the centered Gaussian process $\Psi(\mathbb{G}_{\sigma})$ on $\widebar{B}_{R}(0)^{k}$ that satisfies
\begin{align*}
    \text{cov}\Psi(\mathbb{G}_{\sigma})(w, \tilde{w})= \int_{V \in \mathbb{S}^{d-1}} W_{2}^{2}(\widebar{\mu}_{k}^{V}, \eta(w)^{V})W_{2}^{2}(\widebar{\mu}_{k}^{V}, \eta(\tilde{w})^{V})\dd \sigma(V) - \mathcal{F}_{k}(\eta(w))\mathcal{F}_{k}(\eta(\tilde{w})). 
\end{align*}
Finally, applying the continuous mapping theorem again for the infinite norm of $\ell^{\infty}(\widebar{B}_{R}(0)^{N})$, we obtain the uniform convergence result that 
$\sqrt{m}\Vert\widebar{\mathcal{F}}_{k}(\eta(w)) - \mathcal{F}_{k}(\eta(w))\Vert_{\ell^{\infty}(\widebar{B}_{R}(0)^{N})}$ converges in distribution to $\Vert \Psi(\mathbb{G}_{\sigma})\Vert_{\ell^{\infty}(\widebar{B}_{R}(0)^{N})}$. Setting $\mathbb{G} = \Psi(\mathbb{G}_{\sigma})$ completes the proof of the second statement is complete.   
\end{proof}

\begin{proof}[Proof of Theorem \ref{theorem: computation}:]
We first prove the convergence of Algorithm \ref{algorithm: bcd} to the local minima of $\widebar{\mathcal{F}}_{k}(\eta(w))$. We then prove the validity of Algorithm \ref{algorithm: abcd} by establishing the convergence of matrix $L_{m} = \sum_{j = 1}^{m}V_{j}V_{j}^{\top}$, analyzed on Proposition \ref{proposition: algorithm error expectation}.

\par\noindent\textbf{Step 1.} By Lemma \ref{lemma: Bobkov}, the $W_{2}$ distance between $\widebar{\mu}_{k}^{V_{j}}$ and $\eta(w)^{V_{j}}$ can be reformulated as
\begin{align*}
    W_{2}^{2}(\widebar{\mu}_{k}^{V_{j}}, \eta(w)^{V_{j}}) = \frac{1}{k}\sum_{q = 1}^{k}\Big(\langle w_{q}, V_{j}\rangle - D_{\nu_{j}^{-1}(q), \,j}\Big)^{2},
\end{align*}
where we remind that $D\in \mathbb{R}^{k \times m}$ is constructed by sorting $y_{j(q)}\langle V_{j}, X_{j(q)}\rangle$ from $k$-NN samples and $\nu_{j}$ is the sorted indices for $\langle w_{q}, V_{j} \rangle$. This implies
\begin{align}\label{eq: discrete objective function}
    \widebar{\mathcal{F}}_{k}(\eta(w)) = \frac{1}{mk}\sum_{j = 1}^{m}\sum_{q = 1}^{k}\Big(\langle w_{q}, V_{j}\rangle - D_{\nu_{j}^{-1}(q), \,j}\Big)^{2}.
\end{align} 
The optimal transport law on the local point $w$ is a quadratic function that only depends on the order of $\langle w_{q}, V_{j}\rangle$. This enables the separation of $\widebar{B}_{R}(0)^{k}$ into finite blocks where the optimal transport law is consistent on each block. Since both blocks and optimal transport laws are finite and the value of $\widebar{\mathcal{F}}_{k}(\eta(\psi_{\ell}))$ decreases during the iteration of Algorithm \ref{algorithm: bcd}, the algorithm stops after finite number of iterations. Thus, the stopping point $\psi_{\ell}$ achieves the minimum of the optimal transport law for the block that the point is contained, on restricted domain $\widebar{B}_{R}(0)^{k}$. This indicates that the iteration stops at the local minima. 

\par\noindent\textbf{Step 2.} Denote the vectors for $1\leq q \leq k$ as
\begin{align*}
    \widebar{V}_{q} = \sum_{j = 1}^{m} D_{\nu_{j}^{-1}(q), \,j} V_{j}. 
\end{align*}
By direct computation, the convex optimization formulation of Algorithm \ref{algorithm: bcd} for each iteration is
\begin{align}\label{eq: minimization psi}
    \psi_{\ell + 1,q} = \argmin_{w_{q}\in \widebar{B}_{R}(0)} \Gamma_{q}(w_{q}) \,\,\,\text{ with }\,\,\, \Gamma_{q}(U) = m^{-1}(U- L_{m}^{-1}\widebar{V}_{q})^{\top}\,L_{m}\,(U - L_{m}^{-1}\widebar{V}_{q})
\end{align}
and for each iteration of Algorithm \ref{algorithm: abcd}, the convex optimization formulation is given as
\begin{align}\label{eq: minimization tilde psi}
    \tilde{\psi}_{\ell + 1,q} = \argmin_{w_{q}\in \widebar{B}_{R}(0)} \tilde{\Gamma}_{q}(w_{q})\,\,\,\text{ with }\,\,\, \tilde{\Gamma}_{q}(U) = d^{-1}(U -m^{-1}d\,\widebar{V}_{q})^{\top}(U - m^{-1}d\,\widebar{V}_{q})
\end{align}
for same values of $\widebar{V}_{q}$ on $1 \leq q \leq k$ since $\psi_{\ell} = \tilde{\psi}_{\ell}$ holds. 
\par\noindent\textbf{Step 3.} Next step is on the uniform convergence of $\Gamma_{q}(U) - \tilde{\Gamma}_{q}(U)$ on $U \in \widebar{B}_{R}(0)$, such that
\begin{align}\label{eq: approximation uniform convergence}
    \mathbb{P}_{\sigma}\big(\,\lim_{m\to\infty}\Vert \Gamma_{q}(U) - \tilde{\Gamma}_{q}(U) \Vert_{\infty} = 0\,\big) = 1.
\end{align}
We derive the Lipschitz continuity of $\Gamma_{q}$ and $\tilde{\Gamma}_{q}$, which then yields the Lipschitz continuity of $\Gamma_{q} - \tilde{\Gamma}_{q}$. Regarding the Lipschitz continuity of $\tilde{\Gamma}_{q}$, letting $\tilde{U}\in \widebar{B}_{R}(0)$, by direct computation,
\begin{align*}
    \vert \tilde{\Gamma}_{q}(U) - \tilde{\Gamma}_{q}(\tilde{U})\vert &= d^{-1}\, \vert\, \Vert U-m^{-1}d\widebar{V}_{q}\Vert_{2}^{2} - \Vert \tilde{U}-m^{-1}d\widebar{V}_{q}\Vert_{2}^{2}\,\vert\\
    &\leq d^{-1} \,(\,\Vert U-m^{-1}d\widebar{V}_{q}\Vert_{2} + \Vert \tilde{U}-m^{-1}d\widebar{V}_{q}\Vert_{2}\,)\,\Vert U - \tilde{U}\Vert_{2}\leq 4R\,\Vert U - \tilde{U}\Vert_{2},    
\end{align*}
since $\Vert U \Vert_{2}\leq R$ and $m^{-1}\Vert \widebar{V}_{q}\Vert_{2}\leq R$ hold by definition. In the case of $\Gamma_{q}$, since the matrix $L_{m}$ is positive definite, there exists eigendecomposition of $L_{m} = Q_{m}^{\top}B_{m}Q_{m}$ with orthonormal matrix $Q_{m}$ and diagonal matrix $B_{m}$. Letting $\tilde{U}\in \widebar{B}_{R}(0)$,
\begin{align*}
    \vert \Gamma_{q}(U) - \Gamma_{q}(\tilde{U})\vert &= m^{-1}\vert \,U^{\top}L_{m}U-2\widebar{V}_{q}^{\top}U - \tilde{U}^{\top}L_{m}\tilde{U}+2\widebar{V}_{q}^{\top}\tilde{U} \,\vert\\
    &\leq m^{-1}\vert (Q_{m}U)^{\top}B_{m}(Q_{m}U) - (Q_{m}\tilde{U})^{\top}B_{m}(Q_{m}\tilde{U})\vert + 2m^{-1}\vert \widebar{V}_{q}^{\top}(U - \tilde{U})\vert\\
    &\leq\sum_{i = 1}^{d} \vert (Q_{m}U)_{i}^{2} - (Q_{m}\tilde{U})_{i}^{2}\vert  + 2R\Vert U-\tilde{U}\Vert_{2}\leq 2R(d+1)\Vert U-\tilde{U}\Vert_{2},
\end{align*}
by the property that the eigenvalues of $L_{m}$ are bounded by $m$. This leads to the bound
\begin{align*}
    \vert (\Gamma_{q}(U) - \tilde{\Gamma}_{q}(U)) - (\Gamma_{q}(\tilde{U}) -\tilde{\Gamma}_{q}(\tilde{U}))\vert &\leq   \vert \Gamma_{q}(U)- \Gamma_{q}(\tilde{U})\vert + \vert \tilde{\Gamma}_{q}(U) -\tilde{\Gamma}_{q}(\tilde{U})\vert\\
    &\leq 2R(d+3)\Vert U - \tilde{U}\Vert_{2},
\end{align*}
so that $\Gamma_{q}-\tilde{\Gamma}_{q}$ is $2R(d+3)$-Lipschitz continuous with the Lipschitz constant independent of $m$.

Now, we follow the approach in the proof of Proposition \ref{proposition: uniform convergence}. Similar to the set $H$ from \eqref{eq: rational number set}, denote the set $H_{q}$ as 
\begin{align*}
    H_{q} = \Big\{\,U \in \widebar{B}_{R}(0)\,\vert\,U_{i}\in \mathbb{Q}, 1\leq i \leq d\,\Big\} 
\end{align*}
and fix $U\in H_{q}$. Then, denoting the function $\widebar{\Gamma}_{q}(U)$ as 
\begin{align*}
    \widebar{\Gamma}_{q}(U) = d^{-1}(U- L_{m}^{-1}\widebar{V}_{q})^{\top}(U - L_{m}^{-1}\widebar{V}_{q}),
\end{align*}
we have the limit 
\begin{align*}
    \lim_{m \to \infty}\vert \Gamma_{q}(U) - \widebar{\Gamma}_{q}(U) \vert &= \lim_{m \to \infty}\vert (U- mL_{m}^{-1}\cdot m^{-1}\widebar{V}_{q})^{\top}\,(m^{-1}L_{m} - d^{-1}I)\,(U - m L_{m}^{-1}\cdot m^{-1}\widebar{V}_{q})\vert \\
    &= 0
\end{align*}
with probability one. The reason is that $m^{-1} L_{m}$ and $mL_{m}^{-1}$ converge almost surely to $d^{-1} I$ and $dI$, respectively, $m^{-1}\Vert \widebar{V}_{q}\Vert_{2}\leq R$, and every matrix multiplication in the formula is continuous. By a similar derivation and triangle inequality, 
\begin{align*}
    \lim_{m \to \infty}\vert \tilde{\Gamma}_{q}(U) - \widebar{\Gamma}_{q}(U) \vert &= \lim_{m \to \infty} d^{-1}\vert \,\Vert U - L_{m}^{-1}\widebar{V}_{q}\Vert_{2}^{2} - \Vert U - m^{-1}d\,\widebar{V}_{q}\Vert_{2}^{2}\, \vert\\
    &\leq \lim_{m \to \infty} d^{-1}\vert \,\Vert U - L_{m}^{-1}\widebar{V}_{q}\Vert_{2} + \Vert U - m^{-1}d\,\widebar{V}_{q}\Vert_{2}\, \vert\,\Vert(mL_{m}^{-1} - dI)\cdot m^{-1}\widebar{V}_{q}\Vert_{2} \\
    &= 0
\end{align*}
holds almost surely. This gives the limit 
\begin{align*}
    \lim_{m \to \infty}\Gamma_{q}(U) - \tilde{\Gamma}_{q}(U) = 0
\end{align*}
almost surely. Since $H_{q}$ is countable, the limit holds for all $U \in H_{q}$ with probability one. 

Finally, consider the general case of $U \in \widebar{B}_{R}(0)$ and the sequence $(\Gamma_{q}(U) - \tilde{\Gamma}_{q}(U))_{m\in \mathbb{N}}$. Since $\Gamma_{q} - \tilde{\Gamma}_{q}$ is $2R(d+3)$-Lipschitz continuous independent of $m$, the sequence $(\Gamma_{q}(U) - \tilde{\Gamma}_{q}(U))_{m\in \mathbb{N}}$ is uniformly equicontinuous, and this sequence convergence uniformly to the continuous function \citep[Proposition 3.2]{hirsch2012elements}, which is zero in this case. This completes the proof of Equation \eqref{eq: approximation uniform convergence}. 

\par\noindent\textbf{Step 5.} By Equation \eqref{eq: approximation uniform convergence}, there exists $M = M(\omega, \xi)>0$ such that for all $m\geq M$, 
\begin{align*}
    \vert \Gamma_{q}(\psi_{\ell+1, q}) - \tilde{\Gamma}_{q}(\psi_{\ell+1, q})\vert \leq k^{-1}\xi \,\,\,\text{ and }\,\,\, \vert \Gamma_{q}(\tilde{\psi}_{\ell+1, q}) - \tilde{\Gamma}_{q}(\tilde{\psi}_{\ell+1, q})\vert \leq k^{-1}\xi,
\end{align*}
with probability at least $1 - k^{-1}\omega$. Then, by the minimization properties in \eqref{eq: minimization psi} and \eqref{eq: minimization tilde psi}, we have 
\begin{align*}
    \tilde{\Gamma}_{q}(\psi_{\ell+1, q})\leq \Gamma_{q}(\psi_{\ell+1, q})+k^{-1}\xi \leq \Gamma_{q}(\tilde{\psi}_{\ell+1, q})+k^{-1}\xi\leq \tilde{\Gamma}_{q}(\tilde{\psi}_{\ell+1, q})+2k^{-1}\xi,
\end{align*}
which yields
\begin{align*}
    \Vert \psi_{\ell+1, q} - m^{-1}d\,\widebar{V}_{q}\Vert_{2}^{2} \leq \Vert \tilde{\psi}_{\ell+1, q} - m^{-1}d\,\widebar{V}_{q}\Vert_{2}^{2} + 2dk^{-1}\xi. 
\end{align*}
Since $\psi_{\ell+1, q}\in \widebar{B}_{R}(0)$ and $\tilde{\psi}_{\ell+1, q}$ is a projection of $m^{-1}d\widebar{V}_{q}$ to the set $\widebar{B}_{R}(0)$, by the property of projection in \citet[Lemma 3.1]{bubeck2015convex}, 
\begin{align*}
    \Vert \psi_{\ell+1, q} - \tilde{\psi}_{\ell+1, q}\Vert_{2}^{2}\leq \Vert \psi_{\ell+1, q} - m^{-1}d\widebar{V}_{q}\Vert_{2}^{2} - \Vert \tilde{\psi}_{\ell+1, q} - m^{-1}d\widebar{V}_{q}\Vert_{2}^{2}\leq 2dk^{-1}\xi
\end{align*}
holds with probability at least $1 - k^{-1}\omega$. Combining this result for all $1\leq q \leq k$ and letting $\xi = \epsilon^{2}/2d$, we have
\begin{align*}
    \mathbb{P}_{\sigma}\big(\, \Vert \tilde{\psi}_{\ell+1} - \psi_{\ell+1}\Vert_{2} \leq \epsilon\,\big) \geq 1 - \omega 
\end{align*}
for $m \geq M(\omega, \xi)$. Letting $m \to \infty$ then gives
\begin{align*}
   \lim_{m\to\infty} \mathbb{P}_{\sigma}\big(\, \Vert \tilde{\psi}_{\ell+1} - \psi_{\ell+1}\Vert_{2} \leq \epsilon\,\big) = 1, 
\end{align*}
and since 
\begin{align*}
    W_{2}(\eta(\psi_{\ell+1}), \eta( \tilde{\psi}_{\ell+1}))\leq k^{-1/2}\Vert \tilde{\psi}_{\ell+1} - \psi_{\ell+1}\Vert_{2},
\end{align*}
the proof is complete.
\end{proof}

\subsection{A justification of the computational cost calculation}\label{sec:cc-calculation}

Section \ref{section: algorithm} gives an analysis of the computation cost of Algorithm \ref{algorithm: abcd}. This section provides a justification for this analysis. To this end, we first introduce the following proposition that quantifies the accuracy of the Monte Carlo approximation. 

\begin{proposition}\label{proposition: random projections}
For all $w\in \widebar{B}_{R}(0)^{N}$ and $m \geq 8\pi R^{4}\epsilon^{-2}$ with $\epsilon > 0$, we have
\begin{align*}
    \mathbb{E}_{\sigma}\Big[\,\Big\vert\widebar{\mathcal{F}}_{k}\circ\eta(w) - \mathcal{F}_{k}\circ\eta(w)\Big\vert\,\Big]\leq \epsilon. 
\end{align*}
\end{proposition}

\begin{proof}[Proof of Proposition \ref{proposition: random projections}:]
We begin with a probability bound, then derive the expectation bound. Since
\begin{align*}
    0\leq W_{2}^{2}(\widebar{\mu}_{k}^{V_{i}}, \mu^{V_{i}})\leq 4R^{2},
\end{align*}
by Hoeffding's inequality, we have 
\begin{align*}
     \mathbb{P}\,(\, \vert \widebar{\mathcal{F}}_{k}(\eta(w)) - \mathcal{F}_{k}(\eta(w))\vert \geq \epsilon\,)\leq 2\exp{-m\epsilon^{2}/8R^{4}}. 
\end{align*}
Then, introducing
\begin{align*}
    A_{m} := \vert \widebar{\mathcal{F}}_{k}(\eta(w)) - \mathcal{F}_{k}(\eta(w))\vert, 
\end{align*}
we have 
\begin{align*}
    \mathbb{E}[A_{m}] \leq \int_{0}^{\infty} \mathbb{P}(A_{m} \geq \epsilon)\,\dd \epsilon \leq 2 \int_{0}^{\infty} \exp{-\frac{m\epsilon^{2}}{8R^{4}}} \dd \epsilon = \sqrt{\frac{8\pi R^{4}}{m}}.  
\end{align*}
The proof is then complete by setting $m \geq 8\pi R^{4}\epsilon^{-2}$. 
\end{proof}

The next proposition analyzes the approximation accuracy of  Algorithm \ref{algorithm: abcd}. 

\begin{proposition}\label{proposition: algorithm error expectation}
    Assume the case of $N = k$, Assumption \ref{assumption: coefficient}, and $\psi_{\ell} = \tilde{\psi}_{\ell}$, where $\psi_{\ell}$ and $\tilde{\psi}_{\ell}$ are the values of the $\ell$-th iterations for Algorithms \ref{algorithm: bcd} and \ref{algorithm: abcd}, respectively. Then, for all 
    \[
    m \geq \max\{d, 2312\pi d^{12}R^{4}\epsilon^{-2}\}
    \] 
    and any $\epsilon>0$, we have
    \begin{align*}
    \mathbb{E}_{\sigma}\Big[\,W_{2}\Big(\eta(\psi_{\ell+1}), \eta(\tilde{\psi}_{\ell+1})\Big) \,\Big]\leq \epsilon. 
    \end{align*}
\end{proposition}
\begin{proof}[Proof of Proposition \ref{proposition: algorithm error expectation}:]
We first derive the concentration inequality for $V_{j}V_{j}^{\top}$, and then obtain the $W_{2}$ expectation bound. 
\par\noindent\textbf{Step 1.} Regarding the random direction $V_{j}$ sampled from the Haar distribution $\sigma$, the expectation of $V_{j}V_{j}^{\top}$ is easily checked to be $d^{-1} I$. Then, introducing 
\begin{align*}
    L_{m} = \sum_{j = 1}^{m} V_{j} V_{j}^{\top},
\end{align*}
by the strong law of large numbers, we have
\begin{align*}
    \mathbb{P}_{\sigma}\big(\,\lim_{m\to\infty} m^{-1}L_{m} = d^{-1} I\,\big) = 1. 
\end{align*}
In addition, $L_{m}U = 0$ for $U \in \mathbb{R}^{d}$ indicates $\langle V_{j}, U\rangle = 0$ for all $1 \leq j \leq m$ by the positive semidefiniteness of $V_{j}V_{j}^{\top}$. Since the vectors $V_{1},\dots,V_{d}$ are linearly independent with probability one by induction, $L_{m}$ is positive definite with probability one for fixed $m \geq d$. It yields that $L_{m}$ is invertible for fixed $m\geq d$, and since $m\in \mathbb{N}$ is countable,    
\begin{align*}
        \mathbb{P}_{\sigma}\big(\,L_{m} \,\,\text{is invertible for all}\,\,m \geq d\,\big) = 1.
\end{align*}
By Cramer's rule, inverse operation of a matrix is continuous on $d^{-1}I$, so that
\begin{align*}
    \mathbb{P}_{\sigma}\big(\,\lim_{m\to\infty} m L_{m}^{-1} = d I\,\big) = 1. 
\end{align*}
Introducing further
\begin{align*}
    J_{m} = m^{-1}L_{m} - d^{-1}I, 
\end{align*}
Hoeffding's inequality then implies, for arbitrary positive constant $\delta_{V}$,
\begin{align*}
    \mathbb{P}_{\sigma}\Big(\max_{1 \leq i, j\leq d} \vert (J_{m})_{ij} \vert \leq \delta_{V}\Big) \geq 1 - 2d^{2}\exp(-m\delta_{V}^{2}/2). 
\end{align*}
\par\noindent\textbf{Step 2.} Next, we recall the notations and properties from the proof of Theorem \ref{theorem: computation}, and derive the final bound. Assume the case of 
\[
\max_{i, j} \vert (J_{m})_{ij} \vert \leq \delta_{V} ~~~{\rm with}~~~ \delta_{V}<1/2d^{2},
\]
where the eigenvalues of $L_{m}$ are lower bounded by $md^{-1} - md\delta_{V}$ and upper bounded by $md^{-1} + md\delta_{V}$. For fixed $U\in\widebar{B}_{R}(0)$, we obtain the bound of $\vert \Gamma_{q}(U) - \tilde{\Gamma}_{q}(U)\vert$ with triangle inequality
\begin{align*}
    \vert \Gamma_{q}(U) - \tilde{\Gamma}_{q}(U)\vert\leq \vert \Gamma_{q}(U) - \widebar{\Gamma}_{q}(U)\vert + \vert \widebar{\Gamma}_{q}(U) - \tilde{\Gamma}_{q}(U)\vert.
\end{align*}
For the first term above, we have
\begin{align*}
    \vert \Gamma_{q}(U) - \widebar{\Gamma}_{q}(U) \vert &= \vert (U- L_{m}^{-1}\widebar{V}_{q})^{\top}\,J_{m}\,(U - L_{m}^{-1}\widebar{V}_{q})\vert\leq \delta_{V} d (1+(d^{-1}-d\delta_{V})^{-1})^{2}R^{2} \leq 9d^{3}R^{2}\delta_{V};
\end{align*}
for the second term above, we have
\begin{align*}
    \vert \tilde{\Gamma}_{q}(U) - \widebar{\Gamma}_{q}(U) \vert &= d^{-1}\,\vert \,\Vert U - L_{m}^{-1}\widebar{V}_{q}\Vert_{2}^{2} - \Vert U - m^{-1}d\,\widebar{V}_{q}\Vert_{2}^{2}\, \vert\\
    &= d^{-1}\,\vert \,\Vert U - L_{m}^{-1}\widebar{V}_{q}\Vert_{2} + \Vert U - m^{-1}d\,\widebar{V}_{q}\Vert_{2}\, \vert\,\Vert \,L_{m}^{-1}\widebar{V}_{q} - m^{-1}d\widebar{V}_{q}\,\Vert_{2}\\
    &\leq d^{-1}\cdot (2+d+(d^{-1}-d\delta_{V})^{-1})R \cdot 2d^{3}\delta_{V}R\\ 
    &= 8d^{3}R^{2}\delta_{V}.
\end{align*}
Considering this bound for $U = \psi_{\ell + 1, q}, \tilde{\psi}_{\ell+1, q}$ and combining it with the minimization properties of \eqref{eq: minimization psi} and \eqref{eq: minimization tilde psi}, we obtain
\begin{align*}
    \Vert \psi_{\ell + 1, q} - \tilde{\psi}_{\ell + 1, q}\Vert_{2}^{2} \leq d(\tilde{\Gamma}_{q}(\psi_{\ell + 1, q}) - \tilde{\Gamma}_{q}(\tilde{\psi}_{\ell + 1, q}))\leq 34d^{4}R^{2}\delta_{V},
\end{align*}
where the first inequality is obtained invoking \citet[Lemma 3.1]{bubeck2015convex}. Therefore, we have
\begin{align*}
    \mathbb{P}_{\sigma}\big(\,\Vert \psi_{\ell+1, q} - \tilde{\psi}_{\ell+1, q}\Vert_{2}^{2}\leq 34d^{4}R^{2}\delta_{V}\,\big) \geq 1 - 2d^{2}\exp{-m\delta_{V}^{2}/2} 
\end{align*}
holds for all $\delta_{V}<1/2d^{2}$. Then, 
\begin{align*}
    \mathbb{E}_{\sigma}[\,\Vert \psi_{\ell+1, q} - \tilde{\psi}_{\ell+1, q}\Vert_{2}^{2}\,] &\leq \int_{0}^{4R^{2}}\mathbb{P}_{\sigma}(\Vert \psi_{\ell+1, q} - \tilde{\psi}_{\ell+1, q}\Vert_{2}^{2} \geq \epsilon)\,\dd \epsilon\\
    & \leq 68d^{6}R^{2}\int_{0}^{\infty}\exp{-\frac{m\delta_{V}^{2}}{2}}\,\dd \delta_{V}\\
    &\leq 34\sqrt{2\pi}d^{6}R^{2}m^{-\frac{1}{2}}.
\end{align*}
This implies
\begin{align*}
    \mathbb{E}_{\sigma}[\,W_{2}^{2}(\eta(\psi_{\ell + 1}),\eta(\tilde{\psi}_{\ell + 1}) \,] \leq k^{-1}\sum_{q = 1}^{k}\mathbb{E}_{\sigma}[\,\Vert \psi_{\ell+1, q} - \tilde{\psi}_{\ell+1, q}\Vert_{2}^{2}\,]\leq 34\sqrt{2\pi}d^{6}R^{2}m^{-\frac{1}{2}}
\end{align*}
and then invoking Jensen's inequality completes the proof. 
\end{proof}

Propositions \ref{proposition: random projections} and \ref{proposition: algorithm error expectation} outline conditions for $m$ in our two main approximations. Considering the case where $R$ is bounded by a polynomial of $d$ and the errors in Propositions \ref{proposition: random projections} and \ref{proposition: algorithm error expectation} are fixed independent of $d$, $m$ is able to be set as a polynomial of $d$. This results in the polynomial time complexity of Algorithm \ref{algorithm: abcd} on $d$, which is $O(mdn\log n+tmk \log k)$. 

\subsection{Proofs of results concerning causal inference}\label{subsection: proofs of causal inference}
\begin{proof}[Proof of Proposition \ref{prop: causal_heavy_tail}:]
By direct computation, we have the inequality
\begin{align*}
    (1 + \Vert (T, t) \Vert_{2})^{2} &= 1 + 2\,\Vert (T, t)\Vert_{2} + \Vert (T, t)\Vert_{2}^{2}\\
    &\geq 1 + \sqrt{2}\,(\Vert T \Vert_{2} + \vert t \vert) + 2\,\Vert T \Vert_{2}\cdot \vert t \vert \geq (1 + \Vert T \Vert_{2})\,(1 + \vert t \vert),
\end{align*}
which leads to the desired inequality
\begin{align*}
    f_{(Z, W)}(T, t) \geq C_{Z}C_{W}(1 + \Vert T \Vert_{2})^{-\kappa}(1 + \vert t \vert)^{-\kappa}\geq C_{Z}C_{W}(1+\Vert (T, t)\Vert_{2})^{-2\kappa}.  
\end{align*}
\end{proof}

\begin{proof}[Proof of Proposition \ref{prop: causal_heavy_tail_2}:]
We have the lower bound of $f_{W_{\epsilon}}(t)$ as
\begin{align*}
    f_{W_{\epsilon}}(t) &\geq \pi^{-1}\epsilon^{-1}(1+ \epsilon^{-2}(M_{W} + \vert t \vert)^{2})^{-1}\geq \pi^{-1}\epsilon^{-1}(1+\epsilon^{-1}(M_{W} + \vert t \vert))^{-2}\\
    &= \pi^{-1}\epsilon\,(\epsilon + M_{W} + \vert t \vert)^{-2} \geq \pi^{-1}\epsilon \max\{1, \epsilon + M_{W}\}^{-2}(1 + \vert t \vert)^{-2},  
\end{align*}
which leads to the bound of $f_{(Z, W_{\epsilon})}(T, t)$ as
\begin{align*}
    f_{(Z, W_{\epsilon})}(T, t) &\geq C_{Z} \pi^{-1}\epsilon\, \max\{1, \epsilon + M_{W}\}^{-2}(1+\Vert T \Vert_{2})^{-\kappa}(1 + \vert t \vert)^{-2}\\
    &\geq C_{Z} \pi^{-1}\epsilon\, \max\{1, \epsilon + M_{W}\}^{-2}(1+\Vert (T, t)\Vert_{2})^{-2\max\{2, \kappa\}}.
\end{align*}
\end{proof}

\subsection{Auxiliary lemmas}\label{subsection: auxillary lemmas}

\begin{lemma}[Lemma 4.2, \cite{bobkov2019one}]\label{lemma: Bobkov}
    Consider two measures $\mu$ and $\tilde{\mu}$ supported on $k$ point masses $a_{1},\dots,a_{k}$ and $\tilde{a}_{1},\dots,\tilde{a}_{k}$ on $\mathbb{R}$ as
    \begin{align*}
        \mu =\frac{1}{k} \sum_{i = 1}^{k} \delta_{a_{i}}\,\, \text{ and }\,\, \tilde{\mu} =\frac{1}{k} \sum_{i = 1}^{k} \delta_{\tilde{a}_{i}}. 
    \end{align*}
    Then, let $a_{(1)}\leq \cdots \leq a_{(k)}$ and $\tilde{a}_{(1)}\leq \cdots \leq \tilde{a}_{(k)}$ be sorted versions of $\{a_i\}_{i\in[k]}$ and $\{\tilde a_i\}_{i\in[k]}$, respectively. The $W_{2}$ distance between $\mu$ and $\tilde{\mu}$ is 
    \begin{align*}
        W_{2}^{2}(\mu, \tilde{\mu}) = \frac{1}{k}\sum_{i=1}^{k}(a_{(i)} - \tilde{a}_{(i)})^{2}.
    \end{align*}
\end{lemma}

\begin{lemma}\label{lemma: expectation}
\begin{itemize}
\item[(1)] Assume $n\to \infty$, $k = O(n^{c})$ with some constant $c = c(d)\in (0,1)$, and Assumption \ref{assumption: bounded covariates}. Then, there exists some constant $C = C(d, \tau_{0})>0$ such that
\begin{align*}
    \mathbb{E}_{X}[\,\gamma(V, k)^{2}\,]\leq C (k/n)^{\frac{2}{d-1}}. 
\end{align*}
\item[(2)] Assume $n \to \infty$, $\rho, \tau \to 0$, $k = O(n^{c})$ with constant $c = c(d, \alpha)$, $0<c<1$, $k/(n\tau\rho^{d-1})\to 0$, and Assumption \ref{assumption: unbounded covariates}. Then, there exists some constant $C = C(d, \tau_{0}, \rho_{0}, C_{\tilde{X}})>0$ such that
\begin{align*}
    \mathbb{E}_{X}&[\,\gamma(V, k)^{2}\,]\leq C\Big\{(k/n\tau)^{\frac{2}{d-1}}+\tau^{\alpha}+\rho\Big\}.
\end{align*}
\end{itemize}
\end{lemma}
\begin{proof} For $V\in\mathbb{S}^{d-1}$ and $0 \leq \theta \leq \sqrt{2}$, denote the set $\tilde{S}_{X}(V, \theta)$ as
\begin{align*}
    \tilde{S}_{X}(V, \theta) = S(V, \theta)\, \cap\, \{\tilde{X}_{1},\dots,\tilde{X}_{n}\}.
\end{align*}
Since $\gamma(V, k)\geq \theta$ induces $\vert \tilde{S}_{X}(V, \theta)\vert \leq k$,
\begin{align*}
    \mathbb{P}_{X}[\,\gamma(V, k)\geq \theta\,] \leq \mathbb{P}_{X}[\,\vert \tilde{S}_{X}(V, \theta)\vert \leq k\,]. 
\end{align*} 

\par\noindent\textbf{Step 1.} We start with the case of Assumption \ref{assumption: bounded covariates}. Viewing $S(V, \theta)$ as a measurable set for the probability measure $\mathbb{P}_{X}$, there exists constant $C_{m, 1} = C_{m, 1}(d, \tau_{0})>0$ such that $\mathbb{P}_{X}[S(V, \theta)]$ is greater or equal to $m_{1} = C_{m, 1}\theta^{d-1}$ \cite[Page 651]{li2011k}. Let the corresponding i.i.d. random variables $Z_{1, \theta}, \dots, Z_{n, \theta}$ with underlying distribution ${\rm Bern}(m_{1})$. Regarding the case of $m_{1} \geq 2k/n$, 
\begin{align*}
    \mathbb{P}_{X}[\,\vert \tilde{S}_{X}(V, \theta)\vert \leq k\,]&\leq \mathbb{P}\,\big(\,\sum_{i=1}^{n}Z_{i, \theta}\leq k\,\big)\leq \mathbb{P}\,\big(\,\sum_{i=1}^{n}[m_{1} - Z_{i, \theta}]\geq nm_{1} - k\,\big)\\
    &\leq \exp{-\frac{3(nm_{1}-k)^{2}}{6nm_{1}(1-m_{1}) + 2(nm_{1}-k)}}\leq \exp{-\frac{3nm_{1}}{32}}
\end{align*}
by Bernstein's inequality. This leads to the expectation bound of $\gamma(V, k)^{2}$ as
\begin{align*}
    \mathbb{E}_{X}[\gamma(V, k)^{2}] &= \int_{0}^{2}\mathbb{P}[\gamma(V, k)^{2}\geq \theta^{2}]\dd \theta^{2}= 2\int_{0}^{\sqrt{2}}\theta \mathbb{P}[\gamma(V, k)\geq \theta]\dd \theta\\
    &\leq \big(\frac{2k}{C_{m, 1}n}\big)^{\frac{2}{d-1}}  + 2\int_{(\frac{2k}{C_{m, 1}n})^{\frac{1}{d-1}}}^{\sqrt{2}}\theta \exp{-\frac{3C_{m, 1}n\theta^{d-1}}{32}}\dd \theta\\
    &\leq \big(\frac{2k}{C_{m, 1}n}\big)^{\frac{2}{d-1}} + 2 \exp{-\frac{3k}{16}}\leq C_{1} \big(\frac{k}{n}\big)^{\frac{2}{d-1}}
\end{align*}
with constant $C_{1} = C_{1}(C_{m, 1})$. 
\par\noindent\textbf{Step 2.} Next is the case of Assumption \ref{assumption: unbounded covariates} with $V\in L(\rho, \tau)$, which leads
\begin{align*}
    S(\,V,\, \min\{\rho/\sqrt{2},\theta\}\,)\, \subset\, S(V, \theta)\,\backslash\, D(\tau).  
\end{align*}
Then, there exists constant $C_{m, 2} = C_{m, 2}(d, \tau_{0}, \rho_{0})>0$ such that $\mathbb{P}_{X}[S(V, \theta)]$ is greater or equal than $m_{2} = C_{m, 2}\tau\min\{\rho/\sqrt{2}, \theta\}^{d-1}$. Considering the i.i.d. random variables $\tilde{Z}_{1,\theta},\dots,\tilde{Z}_{n, \theta}$ having the underlying distribution ${\rm Bern}(m_{2})$, and assuming $m_{2}\geq 2k/n$, by the similar application on Bernstein's inequality
\begin{align*}
    \mathbb{P}_{X}[\,\vert \tilde{S}_{X}(V, \theta)\vert \leq k\,]\leq \mathbb{P}\,\big(\,\sum_{i=1}^{n}\tilde{Z}_{i, \theta}\leq k\,\big)\leq \exp{-\frac{3nm_{2}}{32}}.
\end{align*}
We separate the expectation of $\gamma(V, k)^{2}$ with two terms
\begin{align*}
    \mathbb{E}_{X}[\gamma(V, k)^{2}] &= 2\int_{0}^{\frac{\rho}{\sqrt{2}}}\theta \mathbb{P}[\gamma(V, k)\geq \theta]\dd \theta+ 2\int_{\frac{\rho}{\sqrt{2}}}^{\sqrt{2}}\theta \mathbb{P}[\gamma(V, k)\geq \theta]\dd \theta,
\end{align*}
where the bounds are derived as
\begin{align*}
    2\int_{0}^{\frac{\rho}{\sqrt{2}}}\theta \mathbb{P}[\gamma(V, k)\geq \theta]\dd \theta & \leq \big(\frac{2k}{C_{m, 2} n\tau}\big)^{\frac{2}{d-1}}  + 2\int_{(\frac{2k}{C_{m, 2} n\tau})^{\frac{1}{d-1}}}^{\sqrt{2}}\theta \exp{-\frac{3C_{m, 2}n\tau\theta^{d-1}}{32}}\dd \theta\\
    &\leq \big(\frac{2k}{C_{m, 2}n\tau}\big)^{\frac{2}{d-1}} + 2 \exp{-\frac{3k}{16}}\leq C_{2} \big(\frac{k}{n\tau}\big)^{\frac{2}{d-1}}
\end{align*}
and
\begin{align*}
    2\int_{\frac{\rho}{\sqrt{2}}}^{\sqrt{2}}\theta \mathbb{P}[\gamma(V, k)\geq \theta]\dd \theta &\leq 2\int_{\frac{\rho}{\sqrt{2}}}^{\sqrt{2}}\theta \exp{-\frac{3C_{m, 2}n\tau\theta^{d-1}}{32}}\dd \theta\\ 
    &\leq 2 \exp{-\frac{3C_{m, 2}n\tau\rho^{d-1}}{2^{\frac{d+9}{2}}}}\leq C_{3} \big(\frac{k}{n\tau}\big)^{\frac{2}{d-1}}
\end{align*}
with some constants $C_{2} = C_{2}(C_{m, 2})$ and $C_{3} = C_{3}(C_{m, 2})$.
\par\noindent\textbf{Step 3.} In the final case of $V\in \mathbb{S}^{d-1}\backslash L(\tau)$, set $\tilde{V}\in L(\tau)$ such that 
\begin{align*}
    \Vert V - \tilde{V}\Vert_{2}\leq C_{\tilde{X}}\tau^{\alpha} + \rho. 
\end{align*}
Letting $\tilde{\gamma} = \gamma(\tilde{V}, k)$, $\widebar{\gamma} = (\tilde{\gamma}^{2} + \Vert V-\tilde{V}\Vert_{2})^{1/2}$, and $U \in S(\tilde{V}, \tilde{\gamma})$, we have
\begin{align*}
    \langle V, U\rangle = \langle V - \tilde{V}, U\rangle + \langle \tilde{V}, U\rangle \geq 1 - \Vert V - \tilde{V} \Vert_{2} - \tilde{\gamma}^{2}, 
\end{align*}
which yields $S(\tilde{V}, \tilde{\gamma}) \subset S(V, \widebar{\gamma})$. By definition of $\gamma(V, k)$ and the result of the previous step, 
\begin{align*}
    \mathbb{E}_{X}[\,\gamma(V, k)^{2}\,]\leq \mathbb{E}_{X}[\,\widebar{\gamma}^{2}\,] \leq \Vert V - \tilde{V}\Vert_{2} + \mathbb{E}_{X}[\,\gamma(\tilde{V}, k)^{2}\,]\leq C_{4}((k/n\tau)^{\frac{2}{d-1}}+\tau^{\alpha}+\rho)
\end{align*}
holds with some constant $C_{4} = C_{4}(d, \tau_{0}, \rho_{0}, C_{\tilde{X}})$. The proof is thus complete.
\end{proof}

\begin{lemma}\label{lemma: vector}
Let $U\in \widebar{B}_{R}(0)$ and $V_{1}, V_{2}\in\mathbb{S}^{d-1}$ with the condition $\langle V_{1}, V_{2}\rangle\geq 1-\gamma^{2}$ for some $0\leq\gamma\leq\sqrt{2}$. Then, there exists a constant $C = C(R)$ such that
    \begin{align*}
        \vert \langle U, V_{2}\rangle - \langle U, V_{1}\rangle\langle V_{1}, V_{2}\rangle \vert\leq C\gamma. 
    \end{align*}
\end{lemma}

\begin{proof}
By direct computation of the objective term,
\begin{align}
    \langle U, V_{2}\rangle - \langle U, V_{1}\rangle\langle V_{1}, V_{2}\rangle &=\langle U, V_{2}\rangle - \langle U, V_{2}\rangle\langle V_{1}, V_{2}\rangle^{2}+\langle U, V_{1}\rangle \langle V_{1}, V_{2}\rangle^{2}- \langle U, V_{1}\rangle\langle V_{1}, V_{2}\rangle\nonumber\\
    &= \langle U, V_{2}\rangle (1 - \langle V_{1}, V_{2}\rangle^{2})+\langle V_{1}, V_{2}\rangle (\langle U, V_{2}\rangle \langle V_{1}, V_{2}\rangle - \langle U, V_{1}\rangle).\label{eq: vector} 
\end{align}
For the first term in \eqref{eq: vector}, we have
\begin{align*}
    \vert 1 - \langle V_{1}, V_{2}\rangle^{2} \vert \leq 2\vert 1-\langle V_{1}, V_{2}\rangle \vert \leq 2\gamma^{2},
\end{align*} 
so that
\begin{align*}
    \vert \langle U, V_{2}\rangle (1-\langle V_{1}, V_{2}\rangle^{2})\vert \leq 2R\gamma^{2}.  
\end{align*}
Next, for the second term of \eqref{eq: vector}, 
\begin{align*}
    \langle U, V_{2} - V_{1}\rangle^{2}\leq \langle U, U\rangle \langle V_{2} - V_{1}, V_{2} - V_{1}\rangle \leq 2R^{2}\gamma^{2}
\end{align*}
holds by Cauchy-Schwarz inequality, which leads to the bound
\begin{align*}
    \vert \langle V_{1}, V_{2}\rangle (\langle U, V_{2}\rangle \langle V_{1}, V_{2}\rangle - \langle U, V_{1}\rangle)\vert &\leq \vert \langle U, V_{2}\rangle \langle V_{1}, V_{2}\rangle - \langle U, V_{1}\rangle\vert\\ 
    &\leq \vert \langle U, V_{2}\rangle \langle V_{1}, V_{2}\rangle - \langle U, V_{1}\rangle\langle V_{1}, V_{2}\rangle\vert +\vert \langle U, V_{1}\rangle \langle V_{1}, V_{2}\rangle - \langle U, V_{1}\rangle\vert \\
    &= \vert \langle U, V_{2}-V_{1}\rangle \langle V_{1}, V_{2}\rangle\vert +\vert \langle U, V_{1}\rangle (\langle V_{1}, V_{2}\rangle - 1)\vert\\
    &\leq \sqrt{2}R\gamma +R\gamma^{2}.
\end{align*}
Combining the bounds of the two terms completes the proof. 
\end{proof}

\bibliographystyle{apalike}
\bibliography{smooth}

\begin{thebibliography}{}

\bibitem[Arjovsky et~al., 2017]{arjovsky2017wasserstein}
Arjovsky, M., Chintala, S., and Bottou, L. (2017).
\newblock Wasserstein generative adversarial networks.
\newblock In {\em International Conference on Machine Learning}, pages
  214--223. PMLR.

\bibitem[Beran, 1993]{beran1993semiparametric}
Beran, R. (1993).
\newblock Semiparametric random coefficient regression models.
\newblock {\em Annals of the Institute of Statistical Mathematics},
  45(4):639--654.

\bibitem[Beran et~al., 1996]{beran1996nonparametric}
Beran, R., Feuerverger, A., and Hall, P. (1996).
\newblock On nonparametric estimation of intercept and slope distributions in
  random coefficient regression.
\newblock {\em The Annals of Statistics}, 24(6):2569--2592.

\bibitem[Beran and Hall, 1992]{beran92}
Beran, R. and Hall, P. (1992).
\newblock Estimating coefficient distributions in random coefficient
  regressions.
\newblock {\em The Annals of Statistics}, 20(4):1970 -- 1984.

\bibitem[Beran and Millar, 1994]{Beran76}
Beran, R. and Millar, P.~W. (1994).
\newblock Minimum distance estimation in random coefficient regression models.
\newblock {\em The Annals of Statistics}, 22(4):1976 -- 1992.

\bibitem[Bobkov and Ledoux, 2019]{bobkov2019one}
Bobkov, S. and Ledoux, M. (2019).
\newblock {\em One-dimensional empirical measures, order statistics, and
  Kantorovich transport distances}, volume 261.
\newblock American Mathematical Society.

\bibitem[Bonhomme and Denis, 2024]{Bonhomme2024}
Bonhomme, S. and Denis, A. (2024).
\newblock Estimating heterogeneous effects: applications to labor economics.
\newblock {\em arXiv preprint arXiv:2404.01495}.

\bibitem[Bonneel et~al., 2015]{bonneel2015sliced}
Bonneel, N., Rabin, J., Peyr{\'e}, G., and Pfister, H. (2015).
\newblock Sliced and {R}adon {W}asserstein barycenters of measures.
\newblock {\em Journal of Mathematical Imaging and Vision}, 51:22--45.

\bibitem[Bonnotte, 2013]{bonnotte2013unidimensional}
Bonnotte, N. (2013).
\newblock {\em Unidimensional and evolution methods for optimal
  transportation}.
\newblock PhD thesis, Universit{\'e} Paris Sud-Paris XI; Scuola normale
  superiore (Pise, Italie).

\bibitem[Bousquet et~al., 2017]{bousquet2017optimal}
Bousquet, O., Gelly, S., Tolstikhin, I., Simon-Gabriel, C.-J., and Schoelkopf,
  B. (2017).
\newblock From optimal transport to generative modeling: the vegan cookbook.
\newblock {\em arXiv preprint arXiv:1705.07642}.

\bibitem[Bubeck, 2015]{bubeck2015convex}
Bubeck, S. (2015).
\newblock Convex optimization: Algorithms and complexity.
\newblock {\em Foundations and Trends in Machine Learning}, 8(3-4):231--357.

\bibitem[Diamond and Boyd, 2016]{diamond2016cvxpy}
Diamond, S. and Boyd, S. (2016).
\newblock {CVXPY}: {A} {P}ython-embedded modeling language for convex
  optimization.
\newblock {\em Journal of Machine Learning Research}, 17(83):1--5.

\bibitem[Dunker et~al., 2019]{Dunker19}
Dunker, F., Eckle, K., Proksch, K., and Schmidt-Hieber, J. (2019).
\newblock {Tests for qualitative features in the random coefficients model}.
\newblock {\em Electronic Journal of Statistics}, 13(2):2257 -- 2306.

\bibitem[Dunker et~al., 2025]{Dunker25}
Dunker, F., Mendoza, E., and Reale, M. (2025).
\newblock Regularized maximum likelihood estimation for the random coefficients
  model.
\newblock {\em Econometric Reviews}, 44(2):192--213.

\bibitem[Fan and Park, 2024]{fan2024minimum}
Fan, Y. and Park, H. (2024).
\newblock Minimum sliced distance estimation in a class of nonregular
  econometric models.
\newblock {\em arXiv preprint arXiv:2412.05621}.

\bibitem[Flamary et~al., 2021]{flamary2021pot}
Flamary, R., Courty, N., Gramfort, A., Alaya, M.~Z., Boisbunon, A., Chambon,
  S., Chapel, L., Corenflos, A., Fatras, K., Fournier, N., Gautheron, L.,
  Gayraud, N.~T., Janati, H., Rakotomamonjy, A., Redko, I., Rolet, A., Schutz,
  A., Seguy, V., Sutherland, D.~J., Tavenard, R., Tong, A., and Vayer, T.
  (2021).
\newblock Pot: Python optimal transport.
\newblock {\em Journal of Machine Learning Research}, 22(78):1--8.

\bibitem[Fournier and Guillin, 2015]{Fournier2015}
Fournier, N. and Guillin, A. (2015).
\newblock On the rate of convergence in {W}asserstein distance of the empirical
  measure.
\newblock {\em Probability Theory and Related Fields}, 162(3):707--738.

\bibitem[Fox et~al., 2016]{fox2016simple}
Fox, J.~T., il~Kim, K., and Yang, C. (2016).
\newblock A simple nonparametric approach to estimating the distribution of
  random coefficients in structural models.
\newblock {\em Journal of Econometrics}, 195(2):236--254.

\bibitem[Fox et~al., 2011]{fox2011}
Fox, J.~T., Kim, K.~I., Ryan, S.~P., and Bajari, P. (2011).
\newblock A simple estimator for the distribution of random coefficients.
\newblock {\em Quantitative Economics}, 2(3):381--418.

\bibitem[Gaillac and Gautier, 2021]{gaillac2021nonparametric}
Gaillac, C. and Gautier, E. (2021).
\newblock Nonparametric classes for identification in random coefficients
  models when regressors have limited variation.
\newblock {\em arXiv preprint arXiv:2105.11720}.

\bibitem[Gaillac and Gautier, 2022]{Gaillac22}
Gaillac, C. and Gautier, E. (2022).
\newblock {Adaptive estimation in the linear random coefficients model when
  regressors have limited variation}.
\newblock {\em Bernoulli}, 28(1):504 -- 524.

\bibitem[Hammer et~al., 1996]{hammer1996trial}
Hammer, S.~M., Katzenstein, D.~A., Hughes, M.~D., Gundacker, H., Schooley,
  R.~T., Haubrich, R.~H., Henry, W.~K., Lederman, M.~M., Phair, J.~P., Niu, M.,
  et~al. (1996).
\newblock A trial comparing nucleoside monotherapy with combination therapy in
  hiv-infected adults with cd4 cell counts from 200 to 500 per cubic
  millimeter.
\newblock {\em New England Journal of Medicine}, 335(15):1081--1090.

\bibitem[Han et~al., 2023]{Han23}
Han, F., Miao, Z., and Shen, Y. (2023).
\newblock Nonparametric mixture {MLE}s under {G}aussian-smoothed optimal
  transport distance.
\newblock {\em IEEE Transactions on Information Theory}, 69(12):7823--7835.

\bibitem[Heckman et~al., 1997]{heckman1997making}
Heckman, J.~J., Smith, J., and Clements, N. (1997).
\newblock Making the most out of programme evaluations and social experiments:
  Accounting for heterogeneity in programme impacts.
\newblock {\em The Review of Economic Studies}, 64(4):487--535.

\bibitem[Heiss et~al., 2022]{heiss2022nonparametric}
Heiss, F., Hetzenecker, S., and Osterhaus, M. (2022).
\newblock Nonparametric estimation of the random coefficients model: An elastic
  net approach.
\newblock {\em Journal of Econometrics}, 229(2):299--321.

\bibitem[Hildreth and Houck, 1968]{hildreth1968some}
Hildreth, C. and Houck, J.~P. (1968).
\newblock Some estimators for a linear model with random coefficients.
\newblock {\em Journal of the American Statistical Association},
  63(322):584--595.

\bibitem[Hirsch and Lacombe, 2012]{hirsch2012elements}
Hirsch, F. and Lacombe, G. (2012).
\newblock {\em Elements of functional analysis}, volume 192.
\newblock Springer Science \& Business Media.

\bibitem[Hoderlein et~al., 2010]{Hoderlein2010}
Hoderlein, S., Klemel{\"a}, J., and Mammen, E. (2010).
\newblock Analyzing the random coefficient model nonparametrically.
\newblock {\em Econometric Theory}, 26(3):804--837.

\bibitem[Hohmann and Holzmann, 2016]{Hohmann2016}
Hohmann, D. and Holzmann, H. (2016).
\newblock Weighted angle {R}adon transform: Convergence rates and efficient
  estimation.
\newblock {\em Statistica Sinica}, pages 157--175.

\bibitem[Holzmann and Meister, 2020]{Holzmann20}
Holzmann, H. and Meister, A. (2020).
\newblock {Rate-optimal nonparametric estimation for random coefficient
  regression models}.
\newblock {\em Bernoulli}, 26(4):2790 -- 2814.

\bibitem[Holzmann and Meister, 2024]{holzmann2024multivariate}
Holzmann, H. and Meister, A. (2024).
\newblock Multivariate root-n-consistent smoothing parameter free matching
  estimators and estimators of inverse density weighted expectations.
\newblock {\em arXiv preprint arXiv:2407.08494}.

\bibitem[Juraska et~al., 2022]{speff2trial}
Juraska, M., Gilbert, P.~B., Lu, X., Zhang, M., Davidian, M., and Tsiatis,
  A.~A. (2022).
\newblock {\em speff2trial: Semiparametric Efficient Estimation for a
  Two-Sample Treatment Effect}.
\newblock R package version 1.0.5.

\bibitem[Kolouri et~al., 2018]{kolouri2018sliced}
Kolouri, S., Rohde, G.~K., and Hoffmann, H. (2018).
\newblock Sliced wasserstein distance for learning {G}aussian mixture models.
\newblock In {\em Proceedings of the IEEE Conference on Computer Vision and
  Pattern Recognition}, pages 3427--3436.

\bibitem[Lewbel and Pendakur, 2017]{lewbel2017unobserved}
Lewbel, A. and Pendakur, K. (2017).
\newblock Unobserved preference heterogeneity in demand using generalized
  random coefficients.
\newblock {\em Journal of Political Economy}, 125(4):1100--1148.

\bibitem[Li et~al., 2011]{li2011k}
Li, S., Mnatsakanov, R.~M., and Andrew, M.~E. (2011).
\newblock K-nearest neighbor based consistent entropy estimation for
  hyperspherical distributions.
\newblock {\em Entropy}, 13(3):650--667.

\bibitem[Lim and Han, 2024]{lim2024smoothed}
Lim, K. and Han, F. (2024).
\newblock Smoothed {NPMLEs} in nonparametric {P}oisson mixtures and beyond.
\newblock {\em arXiv preprint arXiv:2406.08808}.

\bibitem[Lin et~al., 2023]{lin2023estimation}
Lin, Z., Ding, P., and Han, F. (2023).
\newblock Estimation based on nearest neighbor matching: from density ratio to
  average treatment effect.
\newblock {\em Econometrica}, 91(6):2187--2217.

\bibitem[Lin and Han, 2024a]{lin2024consistency}
Lin, Z. and Han, F. (2024a).
\newblock On the consistency of bootstrap for matching estimators.
\newblock {\em arXiv preprint arXiv:2410.23525}.

\bibitem[Lin and Han, 2024b]{lin2024failure}
Lin, Z. and Han, F. (2024b).
\newblock On the failure of the bootstrap for {C}hatterjee's rank correlation.
\newblock {\em Biometrika}, 111(3):1063--1070.

\bibitem[Liutkus et~al., 2019]{liutkus2019sliced}
Liutkus, A., Simsekli, U., Majewski, S., Durmus, A., and St{\"o}ter, F.-R.
  (2019).
\newblock Sliced-{W}asserstein flows: Nonparametric generative modeling via
  optimal transport and diffusions.
\newblock In {\em International Conference on Machine Learning}, pages
  4104--4113. PMLR.

\bibitem[Marcus and Gordon, 1971]{marcus1971extension}
Marcus, M. and Gordon, W.~R. (1971).
\newblock An extension of the {M}inkowski determinant theorem.
\newblock {\em Proceedings of the Edinburgh Mathematical Society},
  17(4):321--324.

\bibitem[Miao et~al., 2024]{Miao24}
Miao, Z., Kong, W., Vinayak, R.~K., Sun, W., and Han, F. (2024).
\newblock {Fisher-Pitman} permutation tests based on nonparametric {P}oisson
  mixtures with application to single cell genomics.
\newblock {\em Journal of the American Statistical Association},
  119(545):394--406.

\bibitem[Pavliotis, 2014]{pavliotis2014stochastic}
Pavliotis, G.~A. (2014).
\newblock {\em Stochastic processes and applications}, volume~60.
\newblock Springer.

\bibitem[Rubin, 1950]{rubin1950note}
Rubin, H. (1950).
\newblock Note on random coefficients.
\newblock In {\em Statistical Inference in Dynamic Economic Models}, volume~10,
  pages 419--421. Wiley New York.

\bibitem[Rudin, 1964]{rudin1964principles}
Rudin, W. (1964).
\newblock {\em Principles of mathematical analysis}, volume~3.
\newblock McGraw-Hill New York.

\bibitem[Santambrogio, 2015]{santambrogio2015optimal}
Santambrogio, F. (2015).
\newblock Optimal transport for applied mathematicians.
\newblock {\em Birk{\"a}user, NY}, 55(58-63):94.

\bibitem[Swamy, 1970]{swamy1970efficient}
Swamy, P.~A. (1970).
\newblock Efficient inference in a random coefficient regression model.
\newblock {\em Econometrica}, 38(2):311--323.

\bibitem[Tanguy et~al., 2024]{tanguy2024properties}
Tanguy, E., Flamary, R., and Delon, J. (2024).
\newblock Properties of discrete sliced {W}asserstein losses.
\newblock {\em Mathematics of Computation}, 94:1411--1465.

\bibitem[Tolstikhin et~al., 2018]{tolstikhin2017wasserstein}
Tolstikhin, I., Bousquet, O., Gelly, S., and Schoelkopf, B. (2018).
\newblock Wasserstein auto-encoders.
\newblock In {\em International Conference on Learning Representations}.

\bibitem[Van~der Vaart, 2000]{van2000asymptotic}
Van~der Vaart, A.~W. (2000).
\newblock {\em Asymptotic statistics}, volume~3.
\newblock Cambridge University Press.

\bibitem[Wolfowitz, 1953]{Wolfowitz53}
Wolfowitz, J. (1953).
\newblock Estimation by the minimum distance method.
\newblock {\em Annals of the Institute of Statistical Mathematics}, 5(1):9--23.

\bibitem[Wolfowitz, 1957]{Wolfowitz57}
Wolfowitz, J. (1957).
\newblock The minimum distance method.
\newblock {\em The Annals of Mathematical Statistics}, 28(1):75--88.

\bibitem[Ye et~al., 2023]{10.1093/biomet/asad045}
Ye, T., Shao, J., and Yi, Y. (2023).
\newblock Covariate-adjusted log-rank test: guaranteed efficiency gain and
  universal applicability.
\newblock {\em Biometrika}, 111(2):691--705.

\end{thebibliography}

\end{document}